\documentclass[10pt]{article}
\usepackage{amsmath, geometry, amssymb, fullpage, amsthm}
\usepackage{graphics,graphicx}
\usepackage{tikz}
\usepackage{multicol}
\usetikzlibrary{matrix,arrows}
\usepackage[bookmarks=false]{hyperref}
\usepackage[utf8]{inputenc}

\def\Q{\mathbb Q}

\def\Z{\mathbb Z}

\def\bA{\mathbb{A}}
 
\def\bC{\mathbb{C}}

\def\bQ{\mathbb{Q}}
\def\bR{\mathbb{R}}

\def\bZ{\mathbb{Z}}

\def\cI{\mathcal{I}}

\def\cK{\mathcal{K}}

\def\cO{\mathcal{O}}
\def\cP{\mathcal{P}}

\def\cS{\mathcal{S}}

\def\fm{\mathfrak{m}}

\def\fp{\mathfrak{p}}

\def\Aut{\operatorname{Aut}}

\def\Gal{\operatorname{Gal}}

\def\Nm{\operatorname{Nm}}

\def\det{\operatorname{det}}

\def\disc{\operatorname{Disc}}
\def\Disc{\operatorname{Disc}}

\def\mod{\operatorname{mod}}

\def\Avg{{\rm Avg}}
\def\Cl{{\rm Cl}}
\def\ds{\displaystyle}

\newcommand{\beas}{\begin{eqnarray*}}
\newcommand{\eeas}{\end{eqnarray*}}

\newcommand{\wt}[1]{\widetilde{#1}}

\theoremstyle{plain}
\newtheorem{theorem}{Theorem}[section]
\newtheorem{thm}{Theorem}
\newtheorem{cor}[thm]{Corollary}
\newtheorem{lemma}[theorem]{Lemma}
\newtheorem{corollary}[theorem]{Corollary}
\newtheorem{proposition}[theorem]{Proposition}
\newtheorem{definition}[theorem]{Definition}
\newtheorem{remark}[theorem]{Remark}

\title{The mean number of $3$-torsion elements in ray class groups of quadratic fields}
\author{Ila Varma}

\begin{document}
\maketitle
\begin{abstract} We determine the average number of $3$-torsion elements in the ray class groups of fixed (integral) conductor $c$ of quadratic fields ordered by absolute discriminant, generalizing Davenport and Heilbronn's theorem on class groups. A consequence of this result is that a positive proportion of such ray class groups of quadratic fields have trivial 3-torsion subgroup whenever the conductor $c$ is taken to be a squarefree integer having very few prime factors none of which are congruent to $1 \bmod 3$. Additionally, we compute the second main term for the number of $3$-torsion elements in ray class groups with fixed conductor of quadratic fields ordered by discriminant.
\end{abstract}

\section{Introduction}
In 1971, Davenport-Heilbronn \cite{dh} determined that the mean number of $3$-torsion elements in the class groups of real (respectively, imaginary) quadratic fields ordered by absolute discriminant is $\frac{4}{3}$ (resp., $2$). In this paper, we determine the average number of $3$-torsion elements in {\em ray class groups} of fixed integral conductor of quadratic fields ordered by their discriminant. More precisely, we prove the following theorem.
 
\begin{thm}\label{main}
Fix a positive integer $c$, and let $m$ denote the number of primes $p \mid c$ such that $p \equiv 1\bmod3$. 
\begin{enumerate}
\item[{\rm (a)}] The average size of the $3$-torsion subgroups in the ray class groups of conductor $c$ of real quadratic fields $K$ ordered by discriminant is: 
	$$\lim_{X\rightarrow \infty} \frac{\ds\sum_{\substack{[K:\bQ] = 2 \\ 0 < \Disc(K) < X}} \#\Cl_{3}(K,c)}{\ds\sum_{\substack{[K:\bQ] = 2 \\ 0 < \Disc(K) < X}} 1} \  \ = \ \ \begin{cases} 
	3^m \cdot\bigg(1 + \displaystyle\frac{1}{3}\cdot\ds\prod_{p \mid c}  1 + \displaystyle\frac{p}{p + 1} \bigg) & \mbox{ if $3 \nmid c$,} \vspace{.05in}\\
	3^m \cdot \bigg(1 + \ds\frac{2}{7}\cdot\ds\prod_{p \mid c} 1 + \ds\frac{p}{p+1}\bigg) & \mbox{ if $3 \mid \mid c$, and } \vspace{.05in}\\
	3^{m+1} \cdot \bigg(1 + \ds\frac{5}{7}\cdot\ds\prod_{p \mid c} 1 + \ds \frac{p}{p+1}\bigg) & \mbox{ if $9 \mid c$.} 
	\end{cases}$$
\item[{\rm (b)}] The average size of the $3$-torsion subgroups in  the ray class groups of conductor $c$ of imaginary quadratic fields $K$ ordered by discriminant is:
	$$\lim_{X\rightarrow \infty} \frac{\ds\sum_{\substack{[K:\bQ] = 2 \\ -X < \Disc(K) < 0}} \#\Cl_{3}(K,c)}{\ds\sum_{\substack{[K:\bQ] = 2 \\ -X < \Disc(K) < 0}} 1} \  \ = \ \ \begin{cases}
	3^m \cdot\bigg(1 + \ds\prod_{p \mid c}  1 + \frac{p}{p + 1} \bigg) & \mbox{ if $3 \nmid c$,}\vspace{.05in} \\
	3^m \cdot \bigg(1 + \ds\frac{6}{7}\cdot \ds\prod_{p \mid c} 1 + \displaystyle \frac{p}{p+1}\bigg) & \mbox{ if $3 \mid \mid c$, and } \vspace{.05in}\\ 
	3^{m+1} \cdot \bigg(1 + \ds \frac{15}{7}\cdot\ds\prod_{p \mid c} 1 + \ds\frac{p}{p+1}\bigg) & \mbox{ if $9 \mid c$.} 
	\end{cases}$$
\end{enumerate}
\end{thm}
Above, we denote by $\Cl_3(K,c)$ the 3-torsion subgroup of the ray class group of conductor $c$ of a quadratic field $K$. Thus, the case $c = 1$ in Theorem \ref{main} recovers Davenport-Heilbronn's theorem on the average number of $3$-torsion elements in the class groups of real and imaginary quadratic fields \cite[Theorem 3]{dh}. For $c > 1$, this is the first result of its kind.

The Cohen-Lenstra heuristics \cite{cl1}, which conjecture asymptotics for the distribution of torsion in class groups of quadratic fields ordered by discriminant, were inspired by the constants appearing in \cite{dh} as well as computations. The analogous distributions for $p$-torsion subgroups in  ray class groups of quadratic fields can be predicted; for example, we expect that the average size of $p$-torsion subgroups in the ray class groups of conductor $c$ coprime to $p$ of real quadratic fields ordered by discriminant is
	\begin{equation} p^{\#\{\ell \mid c: \ \ell \equiv 1 \bmod{p}\}} \cdot\left(1 + \frac{1}{p}\cdot\prod_{{\ell\mid c}\atop{\ell \equiv \pm 1 (p)}}\left(1 + \frac{p-1}{2} \cdot \frac{\ell}{\ell + 1}\right)\right),\end{equation}
	where $\ell$ runs over primes dividing $c$. Such a generalization of the Cohen-Lenstra heuristics predicting the full distribution for ray class groups of families of fixed degree fields would give a more explicit and tangible description of the maximal abelian extension over number fields other than $\bQ$ and imaginary quadratic fields (see Pagano-Sofos \cite{PS}.)

While the mean values in Theorem \ref{main} do depend on the conductor $c$, if we instead average over quadratic fields with discriminant coprime to the conductor, we obtain different constants that only depend on the number of primes dividing $c$. Note that when averaging over a family of quadratic fields defined by prescribed splitting conditions at a finite set of primes, the average size of the 3-torsion subgroups in ray class groups only changes when the set of primes includes prime factors of $c$. This gives the expected generalization of the case when $c = 1$, in which the mean values do not depend on the family one averages over (see Corollary 4 in \cite{bv3}).

\begin{thm}\label{2}
Fix a positive integer $c$ with $n$ distinct prime factors, and let $m$ denote the number of distinct primes $p \mid c$ that are congruent to $1 \bmod 3$. When quadratic fields are ordered by their absolute discriminant:
\begin{enumerate}
\item[{\rm (a)}] The average number of $3$-torsion elements in the ray class groups of conductor $c$ of real quadratic fields with discriminant coprime to $c$ is:
$$\lim_{X\rightarrow \infty} \frac{\ds\sum_{\substack{[K:\bQ] = 2 \\ (\Disc(K),c) = 1\\0 < \Disc(K) < X}} \#\Cl_{3}(K,c)}{\ds\sum_{\substack{[K:\bQ] = 2 \\(\Disc(K),c) = 1 \\0 < \Disc(K) < X }} 1} \  \ = \ \  \begin{cases} 3^m\cdot\left(1 + \ds\frac{2^{n}}{3}\right) & \mbox{ if $3 \nmid c$,} \vspace{.05in}\\
	3^{m} \cdot \left(1 + \ds\frac{2^{n-1}}{3}\right) & \mbox{ if $3 \mid \mid c$, and }\vspace{.1in} \\
	3^{m+1} \cdot \left(1 + \ds{2^{n-1}}\right) & \mbox{ if $9 \mid c$}.
	\end{cases}
$$
\item[{\rm (b)}] The average number of $3$-torsion elements in the ray class groups of conductor $c$ of imaginary quadratic fields with discriminant coprime to $c$ is: $$\lim_{X\rightarrow \infty} \frac{\ds\sum_{\substack{[K:\bQ] = 2 \\ (\Disc(K),c) = 1 \\ -X < \Disc(K) < 0}} \#\Cl_{3}(K,c)}{\ds\sum_{\substack{[K:\bQ] = 2 \\ (\Disc(K),c) = 1\\-X < \Disc(K) < 0}} 1} \  \ = \ \ \begin{cases} 3^m\cdot\left(1 + 2^{n}\right) & \mbox{ if $3\nmid c$, } \vspace{.1in} \\
	3^{m} \cdot \left(1 + {2^{n-1}}\right) & \mbox{ if $3 \mid \mid c$, and } \vspace{.1in}\\
	3^{m+1} \cdot \left(1 + \ds3\cdot {2^{n-1}}\right) & \mbox{ if $9 \mid c$}.
	\end{cases}
$$
\end{enumerate}
\end{thm}

More generally, we find that the mean size of the 3-torsion subgroup in ray class groups of conductor $c$ of quadratic fields defined by prescribing splitting conditions at a finite set of primes only depends on the specific primes $p$ dividing $c$ that are allowed to ramify in the family of quadratic fields and the number of primes $p \mid c$ that are required to remain unramified in the family (see Theorem \ref{maingen}). Such generalizations often shed light on the mass formulas that dictate class group asymptotics (see e.g.\ Theorem 3 of \cite{bv3}. We further expect Theorem \ref{maingen} to have concrete applications to questions surrounding the arithmetic properties of fundamental units (see e.g., \cite{Eisenstein,Stevenhagen})

The averages obtained in \cite{dh} implied that a positive proportion of real (respectively, imaginary) quadratic fields have class number indivisible by 3 when ordered by discriminant, and this result was slightly refined by Nakagawa-Horie \cite{nh} in order to prove that there are infinitely many hyperelliptic curves over $\bQ$ of a given genus with no integral points. In joint work with Bhargava, we prove the following (as a consequence of Corollary 4 in \cite{bv3}): 

\begin{thm}[Bhargava-Varma \cite{bv3}] \label{bv3} Let $S_+ \cup S_0 \cup S_-$ be a disjoint union of finite sets of primes. There are infinitely many real (respectively, imaginary)  quadratic fields  $K$ with class number indivisible by $3$ such that  $K$ is ramified at each prime of $S_0$, inert at each prime of $S_-$, and split at each prime of $S_+$. 
\end{thm}
This result and its generalizations (see Wiles \cite{Wiles} in conjunction with Beckwith \cite{Beckwith}) have been utilized to imply unconditional versions of modularity lifting theorems in the residually reducible case as in Skinner-Wiles \cite{SkinnerWiles}. Furthermore, they are required in proving the nonvanishing of critical values of $L$-functions for positive proportions of quadratic twist families of elliptic curves with rational $p$-torsion points (see Vatsal \cite{vatsal}) and have applications to proving cases of the weak Goldfeld conjecture (see \cite{krizli}).

In this article, we show that the mean values in Theorem \ref{main} also imply that a positive proportion of quadratic fields have trivial $3$-torsion subgroups in their ray class groups for certain conductors $c$, generalizing Theorem \ref{bv3}. 

\begin{cor}\label{3} 
Assume $c$ is equal to an odd prime number not congruent to $1 \bmod 3$ (in the real quadratic case, also consider those conductors $c$ that are a product of two distinct primes that are not congruent to $1 \bmod 3$). Additionally, let $S_+ \cup S_- \cup S_0$ be a disjoint union of finite sets of prime numbers, none of which contain the primes dividing $c$. There are infinitely many real (respectively, imaginary) quadratic fields $K$ that are split at each prime in $S_+$, inert at each prime in $S_-$, ramified at each prime in $S_0$, and have trivial 3-torsion subgroups in their ray class groups of conductor $c$. 
\end{cor}

Finally, we may apply the methods of Taniguchi and Thorne \cite{TaniguchiThorne} to compute the second main term for the mean number of $3$-torsion elements in ray class groups of quadratic fields ordered by absolute discriminant. More precisely, we prove the following refinement of Theorem \ref{main}. Let $\Cl_3(K_2,c)$ denote the 3-torsion subgroup of the ray class group of conductor $c$ for the quadratic field $K_2$.

\begin{thm}\label{4}
For any positive integer $c$ coprime to $3$, let $m$ denote the number of distinct primes dividing $c$ that are congruent to $1 \bmod 3$. When quadratic fields are ordered by absolute discriminant: 
\begin{eqnarray*}
\ds \sum_{0 < \Disc(K_2) < X} \#\Cl_{3}(K_2,c)  
  &=& 3^m \cdot \Bigg(1 +  \frac{1}{3}\cdot \prod_{p \mid c}\left( 1 + \frac{p}{p+1}\right)\cdot \sum_{0 < \Disc(K_2) < X} 1 \\
		 & + & { \displaystyle{\frac{\sqrt{3}\zeta(2/3)\Gamma(1/3)(2\pi)^{1/3}}{15\Gamma(2/3)\zeta(2)}}\cdot\prod_p\left(1 - \frac{p^{1/3} + 1}{p(p+1)}\right)\cdot \prod_{p \mid c} \bigg(1 + \frac{p(1-p^{1/3})}{1 -\frac{p(p+1)}{p^{1/3} + 1}}\bigg)}\cdot X^{5/6}\Bigg) \\
		 & + & \ O_{\epsilon,c}(X^{5/6 - 7/138 + \epsilon}), \mbox{ and } \\
\ds \sum_{-X < \Disc(K_2) < 0} \#\Cl_{3}(K_2,c)  
  &=&  3^m \cdot \Bigg(1 +  \prod_{p \mid c}\left( 1 + \frac{p}{p+1}\right)\cdot \sum_{-X < \Disc(K_2) < 0} 1 \\
		 & + & {\displaystyle{\frac{\,\zeta(2/3)\Gamma(1/3)(2\pi)^{1/3}}{5\Gamma(2/3)\zeta(2)}}\cdot\prod_p\left(1 - \frac{p^{1/3} + 1}{p(p+1)}\right)\cdot \prod_{p \mid c} \bigg(1 + \frac{p(1-p^{1/3})}{1 -\frac{p(p+1)}{p^{1/3} + 1}}\bigg)}\cdot X^{5/6}\Bigg) \\
		 & + & \ O_{\epsilon,c}(X^{5/6 - 7/138 + \epsilon}).\end{eqnarray*}
\end{thm} 

To derive the result for $3$-torsion ideal classes in class groups of quadratic fields, Davenport and Heilbronn first provide asymptotic formulae for the number of cubic fields having bounded discriminant and sieve to count the {\em nowhere totally ramified} cubic fields. These are degree $3$ extensions $K_3$ of $\bQ$ in which any rational prime $p$ that ramifies is of the form $(p) = \fp_1^2\fp_2$ where $\fp_1$ and $\fp_2$ are two distinct primes of $K_3$. They prove that the number of nowhere totally ramified cubic fields having bounded discriminant determines the number of 3-torsion ideal classes in quadratic fields with the same bound on their discriminant, and so they deduce the above theorem.

To prove Theorems \ref{main}, \ref{2} and \ref{4}, we first prove a new parametrization theorem that determines the number of $3$-torsion elements of ray class groups of quadratic fields with bounded discriminant in terms of the number of appropriate (pairs of) cubic fields with related bounds on their discriminants. We then employ a generalization of Davenport and Heilbronn's asymptotics for cubic fields given in Bhargava-Shankar-Tsimerman \cite{simple} by simplifying the asymptotic count of the relevant pairs of cubic fields. It is important to note that Theorem \ref{main} would not follow from the original asymptotics given in \cite{dh}. 

We begin this article by fixing an conductor $c$ and a quadratic field $K_2$ in Section 2 in order to compare the number of $3$-torsion ideal classes in the ray class group of $K_2$ of conductor $c$ to the number of pairs of cubic fields whose discriminants satisfy certain $c^2$-divisibility conditions (see Theorem \ref{raypara}). We additionally study the action of $\Gal(K_2/\bQ)$ on this $3$-torsion subgroup in order to relate the number of $3$-torsion ideal classes with a fixed action of $\Gal(K_2/\bQ)$ to certain singleton cubic fields whose discriminants satisfy similar $c^2$-divisibility conditions. In Section 3, we recall and employ results of \cite{simple} that compute the density of discriminants of cubic fields satisfying certain {\em acceptable} local specifications. This allows us to determine in Section 4 the mean size of the $3$-torsion subgroups in an eigenspace of the ray class groups of quadratic fields $K_2$ for the nontrivial action of $\Gal(K_2/\bQ)$. We are then able to conclude Theorems \ref{main} and \ref{2} as well as Corollary \ref{3} in Section 5 by studying the 3-torsion elements in ray class groups of quadratic fields $K_2$ that are fixed by $\Gal(K_2/\bQ)$. Finally, in Section 6 we prove Theorem \ref{4} by computing the second main term for the average number of $3$-torsion elements in ray class groups of fixed conductor of quadratic fields with bounded discriminant, building on work of \cite{TaniguchiThorne}. 

\section{Parametrization of $3$-torsion elements in ray class groups of quadratic fields}

We begin by describing a bijection between index-$3$ subgroups of ray class groups of quadratic fields and certain pairs of cubic fields. This will allow us to determine the number of $3$-torsion elements in ray class groups of fixed conductor of quadratic fields using a generalization given in \cite{simple} of Davenport-Heilbronn's asymptotic formulae on the density of discriminants of cubic fields. 
 
\subsection{Ray class groups and fields}
   First, we recall the definition of the ray class group of a number field $K$. Because we will eventually range over all quadratic fields, we only consider ray class groups whose finite part of the modulus is integral (so that it can be fixed independently of the quadratic field). Additionally, because ramification at infinity only affects the size of the $2$-torsion subgroup in the (narrow) ray class groups, we work with ray class groups with trivial infinite part of the modulus. Under these restrictions, we refer to the rational positive generator of the modulus as the {\em conductor}. 
   
   Fix $c \in \bZ$, and let $\cI_c(\cO_{K})$ denote the subgroup of fractional ideals of $\cO_{K}$ generated by prime ideals coprime to $c\cO_{K}$. Additionally, let $\cP_{1,c}(\cO_{K})$ denote the subgroup of principal ideals $(\alpha)$ such that $\alpha \equiv 1 \bmod{c\cO_{K}}$. We then define the {\em ray class group of conductor $c$}  as the quotient
\begin{equation}
\Cl(K,c) := \cI_c(\cO_{K})/\cP_{1,c}(\cO_{K}).
\end{equation}
In this notation, the ideal class group of a field $K$ is denoted $\Cl(K,1)$. Additionally, let $\Cl_{p}(K,c)$ denote the $p$-torsion subgroup of $\Cl(K,c)$ for any prime $p$. 

There is another (equivalent) definition of $\Cl(K,c)$ as a quotient of the ideles. More precisely, let $\bA_K^\times$ denote the ideles of $K$, and for any $\cO_K$-prime $\fp \mid c$, define
	$$W_{c}(\fp) = 1 + \fm_{\fp}^{c(\fp)},$$
where $\fm_{\fp}$ is the maximal ideal of $\cO_{K_{\fp}}$ and $c(\fp)$ denotes the largest power of $\fp$ which contains $c\cO_K$. Let
	$$W_c = \prod_{\fp\mid c} W_c(\fp) \times \prod_{\fp \nmid c} \cO_{K_\fp}^\times.$$
We can then define
\begin{equation} \Cl(K,c) := \bA_K^\times/K^\times\cdot W_c.
\end{equation}
The fact that these two definitions are equivalent can be found, e.g. in Milne \cite{milne}. 

Let $K(c)$ denote the {\em ray class field of conductor $c$} of $K$, which is characterized as the unique abelian extension of $K$ such that the Artin map provides an isomorphism between $\Cl(K,c)$ and $\Gal(K(c)/K)$. It is well-known that every finite abelian extension is contained in some ray class field. The {\em conductor} of a finite abelian extension $L/K$ is defined to be the conductor of the smallest ray class field of $K$ that $L$ lies in (note that if $c \mid c'$, then $K(c) \subset K(c')$). Additionally, it is true that any prime $\fp$ of $\cO_{K}$ that ramifies in $L$ must divide $c$.

The importance of the conductor of a finite abelian extension is that it determines exactly which primes ramify. We next show that conductors of cubic cyclic extensions over a quadratic field are squarefree away from $3$ and never divisible by 27.

\begin{lemma}\label{conductor} Fix a integer $c \in \bZ$ with prime factorization $c = 3^k\cdot\prod_{j=1}^n p_j^{k_j}$. 
\begin{enumerate}
\item[\rm{(a)}] If $k = 0$, any cubic cyclic extension of a quadratic field $K$ that is unramified away from primes dividing $c$ is contained in the ray class field $K\bigl(\prod_{j=1}^n p_j\bigr)$.
\item[\rm{(b)}] If $k > 0$, any cubic cyclic extension of a quadratic field $K$ that is unramified away from the primes dividing $c$ contained in the ray class field $K\bigl(9\cdot \prod_{j=1}^n p_j\bigr).$
\end{enumerate}
\end{lemma}

\begin{proof}
Part (6)  of Theorem 9.2.6 in \cite{cohen} implies that the conductor $f$ of a cubic extension $L$ over $K$ is squarefree away from $3$. We deduce part (a) by noting that $f \mid \prod_{j=1}^n p_j$ since $L$ cannot ramify at any prime which is coprime to $c$. 

If $3 \mid f$ and $\fp$ is a prime ideal of $K$ above 3, then $(\cO_{K_{\fp}}^\times)^3$ contains $1 + 9 \cO_{K_{\fp}} = (1 + 3 \cO_{K_{\fp}})^3$. Using the definition $\Cl(K,f) = \bA_K^\times/K^\times W_f$, there is an index-3 subgroup of $\bA_K^\times$ corresponding to the extension $L$. Since any such index-3 subgroup of $\bA_K^\times$ will contain the cubes $(\cO_{K_{\fp}}^\times)^3$, we deduce that $9 \mid f$, but $27 \nmid f$, and we can then combine this fact with part (a) to deduce part (b). \end{proof} 

Lemma \ref{conductor} implies that the minimality restriction on conductors of cubic extensions of quadratic fields requires that such conductors are integers $c$ which are squarefree away from 3 and additionally, $27 \nmid c$. We next study the relationship between conductors and discriminants of cubic extensions. 

\begin{lemma}\label{cubic} 
\begin{enumerate}
\item[\rm (a)] Let $K$ be a quadratic field. If $L$ is a non-Galois cubic field such that the compositum $LK$ is Galois over $\bQ$, then $\Disc(L) = \Disc(K)f^2$, where $f$ is equal to the conductor of $LK$ over $K$. 
\item[\rm (b)] If $L$ is a Galois cubic field and $\Disc(L) = f_0^2$, then $f_0 = 3^e \cdot p_1 \cdot \hdots \cdot p_m$ where $e = 0$ or $2$ and each $p_i$ denotes a distinct prime satisfying $p_i \equiv 1 \bmod 3$ for all $i$. Additionally, $L \subset \bQ(f_0)$. 
\end{enumerate}
\end{lemma}
\begin{proof} Part~(a) follows from Theorem 9.2.6(4) in \cite{cohen}. Part~(b) follows from class field theory (see \cite{cohn}).
\end{proof}

Next, we explicitly determine cubic fields that lie inside the normal closure (over $\bQ$) of a cubic cyclic extension of a quadratic field $K$. We show that the quantity of such cubic fields can be used to compute the number of the $3$-torsion elements in the ray class groups of $K$.

\subsection{Index-$3$ subgroups of ray class groups of quadratic fields}
For the remainder of the section, fix a conductor $c$ as described by Lemma \ref{conductor}, i.e. let $c$ be a positive integer which is squarefree away from 3, and $27 \nmid c$. We describe the relationship between index-$3$ subgroups of $\Cl(K_2,c)$ for a quadratic field $K_2$ and certain pairs of cubic field. To do so, we must first introduce some notation. Call an integer $d \in \bZ$ {\em fundamental} if it is the discriminant of some quadratic field.
\begin{definition} We say that a pair of fields $(K^+,K^-)$ is {\bf $\mathbf c$-valid} if:
\begin{itemize}
\item $K^+ = \bQ$ or a Galois cubic field with $\disc(K^+) \mid c^2$, and
\item $K^- = \bQ$ or a non-Galois cubic field with $\disc(K^-) = df^2$ where $d \in \bZ$ is fundamental and $f \mid c$.
\end{itemize} Two $c$-valid pairs $(K^+,K^-)$ and $(M^+,M^-)$ are {\bf isomorphic} if both $K^+ \cong M^+$ and $K^- \cong M^-$.\end{definition}
\noindent To see an explicit example,  let $c = 7$, and take $K^+ = \bQ(\zeta_7 + \zeta_7^{-1})$, where $\zeta_7$ denotes a 7th root of unity. If $\theta$ denotes a root of $f(x) = x^3 - x^2 + 5x + 1$, then $K^- = \bQ(\theta)$ totally ramifies at $7$. Since $\Disc(K^+) = 49$ and $\Disc(K^-) = -3\cdot2^2\cdot 7^2$, it follows that $(K^+,K^-)$ is a $7$-valid pair. 

The only $1$-valid pairs have $K^+ = \bQ$, and $K^- = \bQ$ or has discriminant equal to the discriminant of a quadratic field. It is straightforward to check that such cubic fields $K^-$ cannot be totally ramified at any prime. 

Below, we give names to certain special classes of $c$-valid pairs.

\medskip
\begin{definition} Let $c$ be a positive integer which is squarefree away from 3, and $27 \nmid c$.  \begin{enumerate} \item[\rm (a)]   The pair $(\bQ,\bQ)$ is the {\bf trivial $c$-valid pair} for any $c$. 
\item[\rm (b)]  For any cyclic cubic field $K_3$ satisfying $\Disc(K_3) \mid c^2$, $(K_3,\bQ)$ is a $c$-valid pair. We refer to $K_3$ as a {\bf (cyclic) $c$-valid cubic field}. 
\item[\rm (c)] For any non-cyclic cubic field $K_3$ whose discriminant can be written as $df^2$ where $f \mid c$ and $d$ is a fundamental, $(\bQ,K_3)$ is a $c$-valid pair. We refer to $K_3$ as a {\bf (non-cyclic) $c$-valid cubic field}. 
\end{enumerate}
\end{definition}

We now state the main result of this section, which describes a correspondence between $c$-valid pairs and index-$3$ subgroups of $\Cl(K_2,c)$. This will allow us to later determine the size of $\Cl_3(K_2,c)$ in terms of $c$-valid cubic fields.

\begin{theorem}\label{raypara} Let $c$ be a positive integer which is squarefree away from 3, and additionally, $27 \nmid c$. There is a natural bijection between pairs $(K_2,G)$ consisting of a quadratic field $K_2$ along with an index-3 subgroup $G$ of $\Cl(K_2,c)$ and  isomorphism classes of non-trivial $c$-valid pairs. 
\end{theorem}

When $c = 1$ and $\Cl_3(K_2) = \Cl_3(K_2,1)$, Theorem \ref{raypara} is simply the bijection used in \cite{dh} between nowhere totally ramified cubic fields and index-3 subgroups of the class groups of quadratic fields (see also \cite{simple}). We prove this generalization by studying prime ramification in (cubic) subfields contained within the Galois closure of an arbitrary cubic cyclic extension $K_6$ over $K_2$. These cubic subfields turn out to be $c$-valid iff $K_6$ is unramified away from $c$.

The goal for the remainder of this section is to prove Theorem \ref{raypara}. We first discuss the Galois theory of an arbitrary cubic cyclic extension of a quadratic number field.

\subsection{Cubic cyclic extensions of quadratic fields}

In order to prove Theorem \ref{raypara}, we first show that for a fixed quadratic field, any cubic cyclic extension of conductor $c$ is determined by a (unique up to isomorphism) non-trivial $c$-valid pair. To find a candidate for this $c$-valid pair, we look within the normal closure (over $\bQ$) of such sextic fields. For any number field $K$, let $\widetilde{K}$ denote its normal closure over $\bQ$.

Fix a quadratic extension $K_2/\bQ$. If $K_6/K_2$ is a cyclic cubic extension, then the Galois group $\Gal(\widetilde{K}_{6}/\bQ)$ is equal to $S_3$, $C_{6}$, or $S_{3} \times C_{3},$ which are the transitive subgroups with order at least $6$ in the wreath product 
\begin{equation} \Gal(K_{6}/K_2) \ \wr \ \Gal(K_2/\bQ) \cong (C_{3} \times C_{3}) \rtimes C_2 \cong S_3\times C_{3}.
\end{equation}
Note that in the first two cases, $K_6$ is already Galois. We have the following field diagram when $K_6 \neq \wt{K_6}$.

\begin{figure}[h!!!!!!!!!!!!!!!!!!!!]\label{fielddiagram}
  \centering
 \begin{tikzpicture}
  \matrix (m) [matrix of math nodes, row sep=1.5em, column sep=1.25em, text height=1ex, text depth=0.35ex] 
{	&   	  &        & \widetilde{K_{6}} &           &       & \\ \\
  	&{K}^+K_2 &		   &{ K_{6}}           &           &K^-K_2 & \\
    &     K^+ & 	   &                   &           & K^-   & \\
    & 	      &  	   &	{K_2}		   &    	   &  	   & \\ \\
    &         &    	   &  \Q     		   &    	   &       & \\ };
\path
(m-1-4) edge node[above] {\tiny 3}(m-3-2)
(m-1-4) edge node[above] {\tiny 3} (m-3-6)
(m-1-4) edge node[left] {\tiny 3}(m-3-4)
(m-3-2) edge node[left] {\tiny 2} (m-4-2)
(m-3-6) edge node[right] {\tiny 2} (m-4-6) 
(m-5-4) edge node[above] {\tiny 3}(m-3-2)
(m-5-4) edge node[above] {\tiny 3}(m-3-6)
(m-5-4) edge node[left] {\tiny 3}(m-3-4) 
(m-7-4) edge node[left] {\tiny 2}(m-5-4)
(m-4-2) edge node[below] {\tiny 3}(m-7-4)
(m-4-6) edge node[below] {\tiny 3}(m-7-4);

\end{tikzpicture}
  \caption{Some subfields of $\wt{K_6}$ when $\Gal(K_6/\bQ) \cong S_3 \times C_3$}
\end{figure}

When $K_6$ is not Galois, denote the subfield of $\wt{K}_6$ fixed by $C_2 \times C_3 \subset S_3 \times C_3$ as $K^-$, and the subfield fixed by $S_3$ as $K^+$. We then have that $\Gal(K^+/\bQ) = C_3$, and $\Gal(\wt{K}^-/\bQ) = S_3$ with $\wt{K}^- = K^-K_2$. It follows that $\wt{K_6} = \wt{K^+K^-}$. We have thus proven the following lemma stating that $K^+$ and $K^-$ determine $K_6$ and vice versa.

\begin{lemma}\label{lemma1} Let $K_6$ denote a cubic cyclic extension over a quadratic field $K_2$. If $K_6$ is not Galois over $\bQ$, then $$\Gal(\wt{K_6}/\bQ) \cong S_3 \times C_3.$$ Additionally, there exists a unique pair of (isomorphism classes of) cubic subfields $K^+$ and $K^-$, where $K^+$ is cyclic and $K^-$ is not Galois such that the normal closure of $K^+K^-$ is equal to $\widetilde{K_6}$. In particular, we can explicitly write 
	$$\wt{K_6} = K^+ K_2 K^-.$$
\end{lemma}

This lemma implies that any degree-18 field with Galois group over $\bQ$ equal to $S_3 \times C_3$ is either determined by a degree 6 non-Galois subfield that is cyclic over a quadratic subextension or equivalently, by the fixed field of $C_2 \times C_3$ and the fixed field of $S_3$. We will use this to prove that the pair of fields denoted in Figure 1 by $K^+$ and $K^-$ make a $c$-valid pair whenever $K_6$ is unramified away from $c$.

\subsection{Ramification in fields with Galois group $S_3 \times C_3$}
We next show that a pair of cubic fields associated to a cubic cyclic extension of conductor $c$ over a fixed quadratic field by Lemma \ref{lemma1} is indeed $c$-valid.  We do so by understanding how ramification in the sextic field determines ramification in the pair and vice versa. We begin by reviewing properties of the discriminants of subfields within a Galois sextic field.

\begin{lemma}\label{tech} Fix a quadratic field $K_2$. Any $c$-valid cubic field $K$ such that $K_2K$ is Galois over $\bQ$ satisfies:
\begin{enumerate}
\item[\rm (a)] $\Disc(K)^2 \mid \Disc(K_2K);$
\item[\rm (b)] $\Nm_K(\Disc(K_2K/K)) \mid \Disc(K_2)^3$;
\item[\rm (c)] $\Disc(K_2K) \mid c^4 \cdot \Disc(K_2)^3$.
\end{enumerate}
\end{lemma}

\begin{proof} There are two different ways we can calculate the discriminant of $K_2K$ using the field towers $K_2K/K_2/\bQ$ and $K_2K/K/\bQ$:
\begin{equation}\label{discformula}
\Nm_{K_2}(\Disc(K_2K/K_2))\cdot \Disc(K_2)^3 = \Disc(K_2K) = \Nm_{K}(\Disc(K_2K/K))\cdot \Disc(K)^2.\end{equation}

\noindent (a) \ 
By the second equality in \eqref{discformula}, we conclude that $\Disc(K)^2 \mid \Disc(K_2K).$

\noindent (b) \ Let $[\beta_1,\beta_2]$ denote an integral basis of $\cO_{K_2}$, and define the $2 \times 2$ matrix $M = [\sigma_i(\beta_j)]_{i,j}$ where $\sigma_i$ run through elements of $\Gal(K_2K/K)$. First, we know that $\det(M)^2 = \Disc(K_2)$, and second, we have that $[\beta_1,\beta_2]$ is a $K$-basis for $K_2K$. This implies that $\det(M)^2 \in \Disc(K_2K/K)$, and hence $$\Disc(K_2K/K) \mid \Disc(K_2)$$ as $\cO_K$-ideals. Part (b) then follows by taking norms.

\noindent (c) \ { The extension $K_2K/K_2$ is abelian, and by Theorem 9.2.6 of \cite{cohen}, it has an integral conductor $f$. It is related to the relative discriminant by $\Disc(K_2K/K_2) = (f\cO_{K_2})^2.$ 
If $K$ is non-Galois, then $\Disc(K) = df^2$ by Lemma \ref{cubic}(a), and so $f \mid c$. If $K$ is Galois, then $\Disc(K) = f^2$ where $f \mid c$.}

From \eqref{discformula}, we have that
	$$\Disc(K_2K) = \Nm_{K_2}(f^2)\cdot \Disc(K_2)^3.$$
Since $f \mid c$, we obtain  $\Disc(K_2K) \mid \Nm_{K_2}(c^2) \cdot \Disc(K_2)^3,$
which implies that
\begin{equation*}\label{lem} \Disc(K_2K) \mid c^4\cdot \Disc(K_2)^3.
\end{equation*}
\end{proof}

\begin{proposition}\label{prop1} Fix a quadratic field $K_2$. Any cubic cyclic extension $K_6$ over $K_2$ of conductor $c$ that is not Galois over $\bQ$ has a $c$-valid pair $(K^+,K^-)$ contained within the normal closure $\wt{K_6}$ satisfying $\widetilde{K^+K^-} = \wt{K_6}$. It is unique up to isomorphism. 
\end{proposition}

\begin{proof}
Given a cubic cyclic extension $K_6$ over $K_2$ of conductor $c$, the candidate $c$-valid pair $(K^+,K^-)$ associated to $K_6$ comes from Lemma \ref{lemma1} and Figure 1. It remains to show that $\disc(K^+) \mid c^2$ and $\frac{\disc(K^-)}{\disc(K_2)} \mid c^2$. To do so, we study the relationship between total ramification in $K^+$ or $K^-$ and ramification in $K_6/K_2$. 

We begin with $K^-$. By Lemma \ref{cubic}(a), the conductor of $K^-K_2/K_2$ is equal to $f$ where $\Disc(K^-) = \Disc(K_2)f^2$. Combined with Lemma \ref{tech}, we obtain $\Disc(K_2)^2f^4 \mid c^4\Disc(K_2)^3$, and so
	$$f^4 \mid c^4\cdot \Disc(K_2).$$
Recall that $\Disc(K_2)$ is squarefree away from $2$, and $2^4 \nmid \Disc(K_2)$ for any quadratic field, so we conclude that $f \mid c$. 

We now turn to $K^+$. Lemma \ref{tech} implies that
	$$\Disc(K^+)^2 \mid c^4 \Disc(K_2)^3.$$
In this case, $\Disc(K^+) = f_0^2$ for some integer $f_0$, and so we obtain
	$$f_0^4 \mid c^4 \cdot \Disc(K_2)^3.$$
If $\Disc(K_2)$ is odd, then it is squarefree by Lemma \ref{cubic}(b). This implies $f_0 \mid c$. 

If $\Disc(K_2)$ is even, but $2 \nmid f_0$, then we also conclude $f_0 \mid c$. If $2 \mid f_0$, assume for the sake of contradiction that $2 \nmid c$. Then $\wt{K_6}$ is unramified over $K_2$ at the primes above $2$. This implies that the ramification degree of $\wt{K_6}/\bQ$ is at most $2$ for $p = 2$, and thus the ramification degree is at most $2$ in $K^+$. Since $K^+$ is cyclic of degree $3$, $2$ is thus unramified in $K^+$, which contradicts $2 \mid f_0$ since $\Disc(K^+) = f_0^2$. Thus, $2 \mid c$, and since $4 \nmid f_0$ by Lemma \ref{cubic}(b), we obtain $f_0 \mid c$. 
\end{proof}

\subsection{Proof of Theorem \ref{raypara}}

We finally return to Theorem \ref{raypara} and give a proof. We begin by giving an explicit description of the map. Let $c$ be as in the statement of Theorem \ref{raypara}. If $(K_2,G)$ is a quadratic field along with an index-$3$ subgroup $G$ of $\Cl(K_2,c)$, then let $K_6$ denote the fixed field in $K_2(c)$ for the subgroup $G$ so that $\Gal(K_6/K_2) = \Cl(K_2,c)/G$. We then have:
\begin{itemize}
\item If $K_6$ is Galois over $\bQ$ and $\Gal(K_6/\bQ) = C_6$, then $(K_2,G)$ corresponds to $(K_3,\bQ)$, where $K_3$ is the cubic subfield of $K_6$;
\item If $K_6$ is Galois over $\bQ$ and $\Gal(K_6/\bQ) = S_3$, then $(K_2,G)$ corresponds to $(\bQ,K_3)$, where $K_3$ is the cubic subfield of $K_6$;
\item If $K_6$ is not Galois over $\bQ$, then $(K_2,G)$ corresponds to $(K^+,K^-)$ as constructed in Lemma \ref{lemma1}.
\end{itemize}

If $K_6$ is Galois, then its cubic subfield $K_3$ can only totally ramify at primes dividing $c$. Indeed, this is clear when $\Gal(K_6/\bQ) = C_6$. When $\Gal(K_6/\bQ) = S_3$, a prime $p$ ramifies in the extension $K_6/K_2$ if and only if $p^2 \mid \frac{\Disc(K_3)}{\Disc(K_2)}$. Thus, $K_3$ is a $c$-valid cubic field in either case. 

When $K_6$ is not Galois, Proposition \ref{prop1} implies that the above map sends $(K_2,G)$ to a $c$-valid pair $(K^+,K^-)$. To prove the other direction, we begin with a $c$-valid pair $(K^+,K^-)$ and construct a cubic cyclic extension over $K_2$ of conductor dividing $c$. Then, the Galois group of $K_2(c)$ over this cubic cyclic extension will be equal to $G$. 

Recall that the compositum $K^+K_2K^-$ is Galois over $\bQ$ of degree $6$ or $18$. If it is degree $6$, then $(K^+,K^-)$ is in fact a non-trivial $c$-valid {\em cubic field}, i.e., exactly one of $K^\pm$ is equal to $\bQ$. Thus, we take $K_6$ to be the compositum $K^+K_2K^-$. In this case, it remains to show that $K_6$ has conductor dividing $c$ over $K_2$, i.e. $K_6 \subset K_2(c)$. If $K^- = \bQ$, then by assumption $\Disc(K^+) = f_0^2$ where $f_0 \mid c$; thus, $K^+ \subset \bQ(f_0) \subset \bQ(c)$, so $K_6 = K^+K_2 \subset K_2(c)$. If $K^+ = \bQ$, then $\Disc(K^-) = \Disc(K_2)f^2$ where $f \mid c$, and so $K^-K_2 = \wt{K^-}$ is a cubic extension of $K_2$ contained in $K_2(f) \subset K_2(c)$ by Lemma \ref{cubic}(a).   

If $K^+K_2K^-$ is degree $18$, it has Galois group equal to $S_3 \times C_3$, and we define $K_6$ to be the fixed field of any non-normal $C_3 \subset S_3 \times C_3$. It remains to show that $K_6$ has conductor dividing $c$ as an extension over $K_2$. We do so by proving that $K_2K^-$ and $K_2K^+$ have conductor dividing $c$ over $K_2$. 

Lemma \ref{cubic}(a) implies that $K_2K^-$ has conductor $f$ where $\Disc(K^-) = \Disc(K_2)\cdot f^2$, and $f \mid c$.  Additionally, if $\Disc(K^+) = f_0^2$, suppose $p$ is a prime such that $p \nmid f_0$. Then $p$ cannot ramify in $K^+$, which implies that $K^+K_2/K_2$ is unramified above $p$. By Lemma \ref{conductor}, we conclude that $K_2K^+/K_2$ has conductor dividing
	$$\begin{cases} \ds\prod_{p \mid f_0} p & \mbox{ if }3 \nmid f_0, \mbox{ or }\vspace{.1in}\\ \ds9\cdot\prod_{3 \neq p \mid f_0} p & \mbox{ if } 3 \mid f_0. \end{cases}$$
If $3 \nmid f_0$, then $f_0$ is squarefree by Lemma \ref{cubic}(b), and so the conductor of $K_2K^+$ over $K_2$ divides $c$ since $f_0 \mid c$. If $3 \mid f_0$, note that $9 \mid\mid f_0$; thus, we altogether obtain that $K_2K^+$ and $K_2K^-$ both have conductor dividing $c$. Since $\wt{K_6} = K^+K_2K^-$, it must have conductor dividing $c$ as an extension over $K_2$, and so $K_6$ as a subextension must also have conductor dividing $c$. 

It is easy to check that two non-isomorphic $c$-valid pairs $(K^+,K^-)$ correspond to distinct non-isomorphic cubic cyclic extensions of $K_2$ of conductor $c$, and thus, they correspond to distinct index $3$-subgroups of $\Cl(K_2,c)$. \hfill $\square$
\\

Using the fact that the number of order-$3$ subgroups is equal to the number of index-$3$ subgroups in a finite abelian group, we directly relate the number of non-trivial $c$-valid pairs to the number of $3$-torsion elements in ray class groups of conductor $c$. 

\begin{corollary}\label{6.1} If $c$ is a positive integer which is squarefree away from 3 and $27 \nmid c$, then
$$\#\Cl_3(K_2,c) = 2 \cdot \#\left\{\begin{array}{c}\mbox{\rm non-trivial $c$-valid pairs of fields $(K^+,K^-)$}\\ \mbox{\rm s.t.~$K^- = \bQ$ or has quadratic resolvent $K_2$ }\end{array}\right\} + 1.$$
\end{corollary}

Recall that the {\em quadratic resolvent} of a non-Galois cubic field $K_3$ is the quadratic subfield of the normal closure $\wt{K_3}$. If $K_3$ has quadratic resolvent $K_2$, then $\Disc(K_2)\mid \Disc(K_3)$ and $\frac{\Disc(K_3)}{\Disc(K_2)}$ is equal to the square of the conductor of $\Gal(\wt{K_3}/K_2)$ by Lemma \ref{cubic}(b).

\subsection{The action of $\Gal(K_2/\bQ)$ on $\Cl_3(K_2,c)$}
We next consider the action of $\Gal(K_2/\bQ)$ on $\Cl_3(K_2,c)$. The number of $c$-valid {\em cubic fields} on the right-hand side of the equality in Corollary \ref{6.1} is related to the sizes of eigenspaces for the action of $\Gal(K_2/\bQ)$. Note that $\Cl_3(K_2,c)$ is a $\Gal(K_2/\bQ)$-module of odd order, and thus we have two well-defined submodules of $\Cl_3(K_2,c)$:
	\begin{eqnarray*}
	\Cl_3^+(K_2,c) & := &\{ [I] \in \Cl_3(K_2,c) : \sigma(I) = I\},\mbox{ and }  \\
	 \Cl_3^-(K_2,c)  & := & \{ [I] \in \Cl_3(K_2,c) : \sigma(I) = J \mbox{ where } [I]^{-1} = [J]\}.
	\end{eqnarray*}
We then have $\Cl_3({K_2,c}) = \Cl_3^+({K_2,c}) \oplus \Cl_3^-({K_2,c}),$ and thus
	\begin{equation}\label{pmdecomp}\#\Cl_3({K_2,c}) = \#\Cl_3^+({K_2,c}) \cdot\#\Cl^-_3({K_2,c}).\end{equation}
\begin{proposition}\label{decomp} Fix a quadratic field $K_2$, and let $c$ be a positive integer which is squarefree away from 3, and $27 \nmid c$. Then:
\begin{eqnarray*}
	(\rm{a}) \ \ \#\Cl_3^+(K_2,c) &=& 2\cdot \#\left\{\mbox{\rm Cyclic $c$-valid cubic fields $K^+$}\right\}+1;\\
	({\rm b}) \ \ \#\Cl_3^-(K_2,c) &=& 2\cdot \#\left\{\mbox{\rm Non-cyclic $c$-valid cubic fields $K^-$ with quadratic resolvent $K_2$}\right\}+1.
\end{eqnarray*}
\end{proposition}	
\begin{proof}
The second part follows from Lemma 1.10 of \cite{nakagawa} and Proposition 35 of \cite{bv3}. To prove the first part, consider some cyclic cubic field $K^+$ unramified away from $c$. By class field theory and the proof of Proposition \ref{raypara}, $K^+K_2/K_2$ corresponds to a index-$3$ subgroup $H$ of $\Cl(K_2,c)^{(3)}$, the 3-Sylow subgroup of $\Cl(K_2,c)$. Since $K^+K_2$ is Galois over $\bQ$, $H$ has an action of $\Gal(K_2/\bQ)$. Artin reciprocity implies that 
	$$\sigma(K^+K_2) = K^+K_2 \quad \Rightarrow \quad \sigma(H) = H.$$ Thus, we write $H = H^+ \oplus H^-$, where $H^{\pm} := \{ [I] \in H : \sigma([I]) = [I]^{\pm}\}$. Let $\Cl^{\pm}(K_2,c)^{(3)}$ be defined analogously. Since $H$ is index $3$, it is clear that $H^+ = \Cl^+(K_2,c)^{(3)}$ or $H^- =\Cl^-(K_2,c)^{(3)}. $

We now show that $H^- = \Cl^-(K_2,c)^{(3)}$, so that $\Cl(K_2,c)^{(3)}/H \cong \Cl^+(K_2,c)^{(3)}/H^+$. For any lift $\tilde{\sigma}$ of $\sigma$ to $\Gal(K^+K_2/K_2)$, Artin reciprocity implies the action of conjugation on $\Gal(K^+K_2/K_2)$ by $\tilde{\sigma}$ corresponds to acting by $\sigma$ on $\Cl(K_2,c)^{(3)}/H$. Since $\Gal(K^+K_2/K_2)$ is isomorphic to $C_6$, $\sigma$ acts trivially on $\Cl(K_2,c)^{(3)}/H$ so $H^- = \Cl^-(K_2,c)^{(3)}$. We then have that the number of index-$3$ subgroups of $\Cl^+(K_2,c)^{(3)}$ is the same as the number of order-$3$ subgroups, which are generated by nontrivial elements of $\Cl_3^+(K_2,c)$. Since powers of an element generate the same subgroup we then deduce the first part.
\end{proof}

We additionally remark that for any quadratic field $K_2$, $\Cl^+_3(K_2,c) = \Cl_3(\bQ,c)$, independent of $K_2$. This is a crucial fact that greatly simplifies the computation for the average size of ray class groups of conductor $c$ when $K_2$ is allowed to vary. We next compute asymptotics for both $\Cl^{\pm}_3(K_2,c)$ by counting the relevant $c$-valid cubic fields as given in Proposition \ref{decomp}.

\section{Counting $c$-valid cubic fields}
The results of the previous section allow us to determine the number of $3$-torsion elements in ray class groups of quadratic fields simply in terms of $c$-valid cubic fields instead of $c$-valid pairs. We first compute the size of $\Cl_3^+(K_2,c)$ for any quadratic field $K_2$ by enumerating the number of cyclic $c$-valid cubic fields.
In order to obtain asymptotics for the size of $\Cl_3^-(K_2,c)$, we then employ the results of \cite{simple} building on those of \cite{dh} for computing the number of cubic fields with bounded discriminant that satisfy certain ramification restrictions. 

\subsection{The size of the $3$-torsion subgroup in ray class groups of $\bQ$} \label{caseg}

As before, let $K_2$ be a quadratic field. In this section, we prove that the number of $\Gal(K_2/\bQ)$-stable elements in the $3$-torsion subgroups of the ray class group of conductor $c$ depends only on the number of distinct primes congruent to $1\bmod{3}$ that divide $c$. More precisely,
\begin{proposition}\label{h+} Let $K_2$ be a quadratic field, and let $c$ be a positive integer. Let the number of distinct prime factors $p_i \mid c$ such that $p_i \equiv 1 \mod 3$ be denoted by $m$. Then $$\#\Cl_{3}^+(K_2,c) = \begin{cases} 3^{m} &\mbox{ if } 9 \nmid c, \mbox{ and } \\ 3^{m+1} & \mbox{ if } 9 \mid c,\end{cases} $$
independent of the quadratic field $K_2$.
\end{proposition}

\begin{proof}
By Proposition \ref{decomp}(a), we can enumerate elements of $\Cl_3^+(K_2,c)$ by counting cyclic $c$-valid cubic fields, i.e., normal cubic extensions of $\bQ$ with discriminant dividing $c^2$. If $K^+$ is such a cyclic field of degree 3, by Lemma \ref{cubic}(b), the conductor of $K^+$ is equal to $c_0 = 3^e\cdot p_1\cdot \hdots \cdot p_m$ where $e = 0$ or $2$, and $p_i$ denotes distinct primes satisfying $p_i \equiv 1 \bmod{3}$ for all $i$. Furthermore, if $e = 0$, then there are $2^{m-1}$ cubic cyclic fields of conductor $c_0$, and if $e = 2$, then there are $2^m$ cubic cyclic fields of conductor $c_0$ (see \cite{cohn}). We must therefore enumerate cyclic $c$-valid cubic fields with discriminant $c_0^2$ where $c_0$ is as above.
 
If $9\nmid c$, let ${c} = 3^e\cdot p_1^{k_1}\cdot \hdots \cdot p_{m}^{k_m}\cdot q_{m+1}^{k_{m+1}}\cdot \hdots\cdot q_n^{k_n}$ where each $p_i$ is a distinct prime congruent to $1 \bmod{3}$, $q_j$ are distinct primes congruent to $2 \bmod{3}$, and $e = 0$ or $1$. 
In conjunction with Proposition \ref{decomp}(a), we obtain $$\#\Cl_{3}^+(K_2,c) = 1 + 2\cdot\left(\sum_{j=1}^m \binom{m}{j}2^{j-1}\right) = 3^m.$$ 
Similarly, if $e \geq 2$, we deduce
	$$\#\Cl_3^+(K_2,c) = 1 + 2\cdot\left(\sum_{j=1}^{m+1} \binom{m+1}{j} 2^j\right) = 3^{m+1}.$$
\end{proof}

\subsection{The asymptotic number of non-cyclic $c$-valid cubic fields}

We want to next determine the asymptotics for the number of cubic fields that are totally ramified at a certain fixed set of primes. Let $\cK^{\rm full}$ denote the set of isomorphism classes of cubic fields, and for any subset $\cK \subseteq \cK^{\rm full}$, define for $i = 0$ or $1$: 
	$$N_3^{(i)}(\cK,X) :=  \#\{K_3 \in \cK \mid 0 < (-1)^i\Disc(K_3) < X\}.$$ 
\begin{theorem} \label{simple} Let $S$ denote a set of primes, and let $n_i = \#\Aut(\bR^{3-2i} \oplus \bC^{i})$ for $i = 0$ or $1$.  
\begin{enumerate}
\item[\rm (a)]  Let $\cK_S$ denote the set of isomorphism classes of cubic fields that are totally ramified exactly at the primes $p \in S$.
$$\lim_{X \rightarrow \infty} \frac{\displaystyle N_3^{(i)}(\cK_S,X)}{X} = \frac{1}{n_i} \cdot \frac{1}{2\cdot\zeta(2)} \cdot \prod_{p\in S} \frac{1}{p(p+1)}$$
\item[\rm (b)] If $3 \in S$, let $\cK^{(3)}_S$ denote the set of isomorphism classes of cubic fields that are totally ramified exactly at $p \in S$ and have discriminant that is {\em not} divisible by 81. 
$$\lim_{X \rightarrow \infty} \frac{\displaystyle N_3^{(i)}(\cK_S^{(3)},X)}{X} = \frac{1}{n_i} \cdot \frac{1}{3\cdot\zeta(2)} \cdot \prod_{p\in S} \frac{1}{p(p+1)}$$
\item[\rm (c)] If $3 \in S$, let $\cK^{(9)}_S$ denote the set of isomorphism classes of cubic fields that are totally ramified exactly at $p \in S$ and have discriminant divisible by 81.
$$\lim_{X\rightarrow\infty}\frac{\displaystyle N_3^{(i)}(\cK^{(9)}_S,X)}{X} = \frac{1}{n_i} \cdot \frac{1}{6\cdot\zeta(2)} \cdot \prod_{p\in S} \frac{1}{p(p+1)}.$$ 
\item[\rm (d)] Let $S'$ be a set of primes containing $S$. Let $\cK_{S,S'}$ denote the set of isomorphism classes of cubic fields that are totally ramified exactly at $p \in S$ and unramified at $p \in S' \smallsetminus S$. 
$$\lim_{X \rightarrow \infty}\frac{\displaystyle N_3^{(i)}(\cK_{S,S'},X)}{X} = \frac{1}{n_i} \cdot \frac{1}{2\cdot\zeta(2)} \cdot \prod_{p \in S' \smallsetminus S} \frac{p}{p+1} \cdot \prod_{p\in S} \frac{1}{p(p+1)}.$$
\item[\rm (e)] Let $S'$ be a set of primes containing $S$. If $3 \in S$, let $\cK^{(9)}_{S,S'}$ denote the set of isomorphism classes of cubic fields $K_3$ that are totally ramified exactly at $p \in S$, unramified at $p \in S' \smallsetminus S$, and $81 \mid \mid \disc(K_3)$. 
$$\lim_{X\rightarrow \infty}\frac{\displaystyle N_3^{(i)}(\cK_{S,S'}^{(9)},X)}{X} = \frac{1}{n_i} \cdot \frac{1}{9\cdot\zeta(2)} \cdot \prod_{p \in S'\smallsetminus S} \frac{p}{p+1} \cdot \prod_{p\in S} \frac{1}{p(p+1)}.$$
\end{enumerate} 
\end{theorem}

\begin{proof} For any prime $p$, let $\Sigma_p^{\rm tr}$ denote the set of all isomorphism classes of maximal cubic \'etale algebras over $\bZ_p$ that are totally ramified. Let $\Sigma_p^{\rm ur}$ denote the set of all isomorphism classes of maximal cubic \'etale algebras over $\bZ_p$ that are unramified. Additionally, let $\Sigma^{\rm ntr}_p$ denote the set of all isomorphism classes of maximal cubic \'etale algebras over $\bZ_p$ that are not totally ramified. When $p = 3$, let $\Sigma_3^{(3)}$ denote the subset of $\Sigma_3^{\rm tr}$ whose discriminant over $\bZ_3$ is not divisible by $81$, and let $\Sigma_3^{(9)}$ denote the subset of $\Sigma_3^{\rm tr}$ whose discriminant is divisible by $81$. Finally, let $\Sigma_3^{(81)}$ denote the subset of $\Sigma_3^{\rm tr}$ whose discriminant is equal to $81$. 

We now describe collections $\Sigma = (\Sigma_p)_p$ of local specifications at each prime $p$ which exactly determine the (maximal orders of) cubic fields contained in $\cK_{S}$, $\cK_{S}^{(3)}$, $\cK_S^{(9)}$, $\cK_{S,S'}$, and $\cK_{S,S'}^{(9)},$ respectively. In each case, $\Sigma$ is an {\em acceptable} collection of local specifications as defined in \cite{simple}. Indeed, we can equivalently define the family of cubic fields (a) $\cK_{S}$, (b) $\cK_{S}^{(3)}$, (c) $\cK_S^{(9)}$, (d) $\cK_{S,S'}$, or (e) $\cK_{S,S'}^{(9)}$ as containing exactly the fraction fields of all maximal orders $R$ for which $R \otimes \bZ_p \in \Sigma_p$ for all $p$ where: 
\begin{multicols}{2}
\begin{enumerate}
\item[\rm (a)] $\Sigma_p = \begin{cases} \Sigma^{\rm ntr}_p & \mbox{ if } p \notin S, \\
 \Sigma_p^{\rm tr} & \mbox{ if } p \in S;
\end{cases}$
\item[\rm (b)] $\Sigma_p = \begin{cases}  \Sigma^{\rm ntr}_p & \mbox{ if } p \notin S, \\
 \Sigma_p^{\rm tr} & \mbox{ if } p \in S \smallsetminus\{3\}, \\
 \Sigma_3^{(3)} & \mbox{ if } p = 3;
\end{cases}$
\item[\rm (c)] $ \Sigma_p =\begin{cases}  \Sigma^{\rm ntr}_p & \mbox{ if } p \notin S, \\
  \Sigma_p^{\rm tr} & \mbox{ if } p \in S \smallsetminus\{3\}, \\
 \Sigma_3^{(9)} & \mbox{ if } p = 3;
\end{cases}$
\item[\rm (d)] 
$\Sigma_p = \begin{cases}   \Sigma^{\rm ntr}_p & \mbox{ if } p \notin S_0, \\
\Sigma_p^{\rm tr} & \mbox{ if } p \in S, \\
 \Sigma_p^{\rm ur} & \mbox{ if } p \in S' \smallsetminus S;
\end{cases}$
\item[\rm (e)]
$\Sigma_p = \begin{cases}   \Sigma^{\rm ntr}_p & \mbox{ if } p \notin S_0, \\
  \Sigma_p^{\rm tr} & \mbox{ if } p \in S \smallsetminus \{3\},\\
  \Sigma_p^{\rm ur} & \mbox{ if } p \in S' \smallsetminus S, \\
 \Sigma_3^{(81)} & \mbox{ if } p = 3.
\end{cases}$
\end{enumerate}
\end{multicols}

Let $N_3^{(i)}(\Sigma,X)$, the number of (isomorphism classes) of maximal cubic rings $R$ such that $R \otimes \bZ_p \in \Sigma_p$ for all $p$ with $0 < (-1)^i \Disc(R) < X$. We asymptotically compute $N_3^{(i)}(\Sigma,X)$ using Theorem 7 in \cite{simple}, which determines the main term in terms of a mass formula whenever $\Sigma$ is defined by an acceptable collection of local conditions. More precisely, they prove
\begin{eqnarray}\label{mass}
\lim_{X\rightarrow \infty} \frac{N_3^{(i)}(\Sigma,X)}{X} &=& \frac{1}{2n_i}\cdot\prod_p\left(\frac{p-1}{p}\cdot\sum_{R \in \Sigma_p} \frac{1}{\Disc_p(R)}\cdot\frac{1}{\#\Aut(R)}\right),
\end{eqnarray}
where $\Disc_p(R)$ denotes the discriminant of $R$ over $\bZ_p$ as a power of $p$. We compute (or combine Lemmas 18, 19, and 32 in \cite{simple} to deduce):
$$
\sum_{R \in \Sigma_p} \frac{1}{\Disc_p(R)}\cdot \frac{1}{\#\Aut(R)} = \begin{cases} 
\frac{p+1}{p} & \mbox{ if }\Sigma_p =\Sigma_p^{\rm ntr},\\
\frac{1}{p^2} & \mbox{ if } \Sigma_p = \Sigma_p^{\rm tr}, \\
{1} & \mbox{ if } \Sigma_p = \Sigma_p^{\rm ur}; \\
\end{cases} \qquad \sum_{R \in \Sigma_3} \frac{1}{\Disc_3(R)}\cdot \frac{1}{\#\Aut(R)} = \begin{cases}
\frac{2}{27} & \mbox{ if } \Sigma_3 = \Sigma_3^{(3)}, \\
\frac{1}{27} & \mbox{ if } \Sigma_3 = \Sigma_3^{(9)}, \mbox{} \\
\frac{2}{81} & \mbox{ if } \Sigma_3 = \Sigma_3^{(81)}. 
\end{cases}
$$
In conjunction with \eqref{mass}, we thus obtain the desired asymptotes in Theorem \ref{simple}. As an example, we give the calculation below in case (e):

\begin{eqnarray*}
\lim_{X\rightarrow \infty} \frac{N_3^{(i)}(\cK_{S,S'}^{(9)},X)}{X}
 &=& \frac{1}{2n_i} \cdot \frac{4}{243} \cdot \prod_{p\notin S'} \frac{p^2-1}{p^2} \cdot\prod_{p \in S' \smallsetminus S} {\frac{p-1}{p}} \cdot\prod_{3 \neq p \in S} \frac{p-1}{p^3}  \\
 &=& \frac{1}{n_i} \cdot \frac{1}{9\cdot\zeta(2)} \cdot \prod_{p \in S' \smallsetminus S} \frac{p}{p+1}\cdot \prod_{p \in S} \frac{1}{p(p+1)}.
\end{eqnarray*}
\end{proof}

\subsection{Prescribing splitting conditions on the quadratic resolvents of $c$-valid cubic fields}

Now, let $\underline{\cS} = (S_+,S_-,S_0)$ be three disjoint sets of primes. We will next consider families $\cK_3(\underline{\cS})$ consisting of all cubic fields whose quadratic resolvent field is in $\cK_2(\underline{\cS})$. (Recall that $\cK_2(\underline{\cS})$ consists of all quadratic fields that split at the primes in $S_+$, remain inert at the primes in $S_-$, and ramify at the primes in $S_0$.)

\begin{theorem} Let $S$ denote a set of primes not containing $3$, and let $n_i = \#\Aut(\bR^{3-2i}\oplus \bC^i)$ for $i = 0$ or $1$. Additionally, let  $\underline{\cS} = (S_+,S_-,S_0)$ be three disjoint sets of primes such that:
\begin{itemize}
\item $S_+ \cap S$ only contains primes congruent to $1 \bmod 3$ .
\item $S_- \cap S$ only contains primes congruent to $2 \bmod 3$.
\end{itemize} As before, let $\cK_{S}$ denote the set of isomorphism classes of cubic fields that are totally ramified exactly at the primes $p \in S$. If $\cK_3^{S}({\underline{\cS}}) = \cK_S \cap \cK_3({\underline{\cS}})$, then we have:
$$\lim_{X \rightarrow \infty} \frac{\displaystyle N_3^{(i)}(\cK_3^{S}({\underline{\cS}}),X)}{X} = \frac{1}{n_i} \cdot \frac{1}{2\cdot\zeta(2)} \cdot\prod_{p \in S} \frac{1}{p(p+1)}\cdot\prod_{p \in S_0} \frac{1}{p+1} \cdot\prod_{p \in S_{\pm}\smallsetminus (S \cap S_{\pm})} \frac{p}{2(p+1)}$$

\end{theorem}

\begin{proof} 
For any prime $p$, let:
\begin{enumerate}
\item $\Sigma_p^{\rm mr}$ denote the set of all (isomorphism classes of) maximal cubic \'etale algebras over $\bZ_p$ that are minimally ramified, i.e., they decompose as $\bZ_p \oplus Q$ where $Q$ is a totally ramified quadratic \'etale algebra over $\bZ_p$; 
\item $\Sigma_p^{+}$ consist of the ring of integers of the unique unramified extension of degree $3$ over $\bQ_p$ as well as the algebra $\bZ_p \oplus \bZ_p \oplus \bZ_p$; 
\item $\Sigma_p^{-} = \{\bZ_p \oplus \bZ_{p^2}\}$ where $\bZ_{p^2}$ denotes the ring of integers of $\bQ_{p^2}$, the unique unramified extension of degree $2$ over $\bQ_p$.
\item $\Sigma_p^{\rm tr +}$ consists of maximal cubic algebras over $\bZ_p$ that are totally ramified and whose quadratic resolvent algebra is contained in $\bQ_p \oplus \bQ_p$. 
\item $\Sigma_p^{\rm tr -}$ consists of maximal cubic algebras over $\bZ_p$ that are totally ramified and whose quadratic resolvent is contained in $\bQ_{p^2}$.
\end{enumerate} 
As before, denote the set of all (isomorphism classes of) maximal cubic \'etale algebras over $\bZ_p$ that are totally ramified as $\Sigma_p^{\rm tr} = \Sigma_p^{\rm tr +} \cup \Sigma_p^{\rm tr -}$, denote the set of all unramified cubic \'etale algebras over $\bZ_p$ as $\Sigma_p^{\rm ur} = \Sigma_p^{+} \cup \Sigma_p^{-}$, and denote the set of all cubic \'etale algebras over $\bZ_p$ that are not totally ramified as $\Sigma_p^{\rm ntr} = \Sigma_p^{\rm ur} \cup \Sigma_p^{\rm mr}$.

If $\Sigma = (\Sigma_p)_p$ denotes the acceptable collection of local specifications defining $\cK_3^{S}({\underline{\cS}})$, then we have:
$$\Sigma_p = \begin{cases} \Sigma^{\rm ntr}_p & \mbox{ if } p \notin S_+ \cup S_- \cup S_0 \cup S, \\
\Sigma_p^{\rm tr} & \mbox{ if } p \in S\smallsetminus (S\cap (S_+\cup S_-)),\\
 \Sigma_p^{\rm mr} & \mbox{ if } p \in S_0, \\
 \Sigma_p^{\pm} & \mbox{ if } p \in S_{\pm}\smallsetminus (S\cap S_{\pm}), \\
 \Sigma_p^{\rm tr \pm} & \mbox{ if } p \in S \cap S_{\pm}.
\end{cases}$$
It is straightforward to determine that
$$
\sum_{R \in \Sigma_p} \frac{1}{\Disc_p(R)}\cdot \frac{1}{|\Aut(R)|} = \begin{cases} 
\frac{1}{p} & \mbox{ if } \Sigma_p = \Sigma_p^{\rm mr},\\
\frac{1}{2} & \mbox{ if } \Sigma_p = \Sigma_p^\pm, \\
{\frac{1}{p^2}} & \mbox{ if } \Sigma_p  = \Sigma_p^{\rm tr +} \mbox{ and } p \equiv 1 \bmod 3. \\
{\frac{1}{p^2}} & \mbox{ if } \Sigma_p  = \Sigma_p^{\rm tr -} \mbox{ and } p \equiv 2 \bmod 3.
\end{cases}
$$
Using \eqref{mass} and the computations following it, we conclude the theorem: 
\begin{eqnarray*}
\lim_{X\rightarrow \infty} \frac{N_3^{(i)}(\cK_3^{S}({\underline{\cS}}),X)}{X}
 &=& \frac{1}{2n_i} \cdot\prod_{p\notin S_\pm\cup S_0\cup S} 
\frac{p^2-1}{p^2} \cdot\prod_{p \in S} \frac{p-1}{p^3}  \cdot \prod_{p \in S_0} \frac{p-1}{p^2}\cdot \prod_{p \in S_\pm \smallsetminus (S \cap S_{\pm})} \frac{p-1}{2p} \\
 &=& \frac{1}{n_i} \cdot \frac{1}{2\cdot\zeta(2)} \cdot \prod_{p \in S} \frac{1}{p(p+1)}\cdot\prod_{p \in S_0} \frac{1}{p+1} \cdot\prod_{p \in S_{\pm}\smallsetminus (S \cap S_{\pm})} \frac{p}{2(p+1)}.
\end{eqnarray*}
(Above, we have abused notation slightly by letting $S_{\pm} = S_+ \cup S_-$.)
\end{proof}

\section{The mean size of $\Cl^-_{3}(K_2,c)$ over families of quadratic fields $K_2$} \label{caseng}

In this section, we begin by computing the average number of $3$-torsion elements in the minus eigenspace of their ray class groups of fixed conductor $c$ in families of quadratic fields ordered by discriminant. We then determine the mean size of $\Cl^-_3(K_2,c)$ over certain subfamilies of quadratic fields $K_2$, namely those defined by local specifications at a finite number of primes. 
We first vary over the quadratic fields whose discriminants are coprime to the choice of conductor $c$ and obtain a different average that only depends on the number of primes dividing $c$. We then average over quadratic fields that have prescribed splitting conditions at a finite number of primes.

\subsection{The average number of 3-torsion elements in the minus eigenspaces of the ray class groups of quadratic fields}

For shorthand, let $$\Avg^{(i)}(\Cl_3^-(c)) :=  \lim_{X \rightarrow \infty} \frac{\ds \sum_{0 < (-1)^i \Disc(K_2) < X} \#\Cl_{3,}^-(K_2,c)}{\ds \sum_{0 < (-1)^i \Disc(K_2) < X} 1}.$$

\begin{proposition} \label{h-} Fix a positive integer $c$, and recall that $n_0 = 6$ and $n_1 = 2$. Then
\begin{enumerate} 
\item[\rm (a)] If $3 \nmid c$, then $\Avg^{(i)}(\Cl_3^-(c)) = \ds 1 + \frac{2}{n_i}\cdot\prod_{p \mid c} \left(1 + \frac{p}{p + 1}\right);$
\item[\rm (b)] If $3 \mid \mid c$, then $\Avg^{(i)}(\Cl_3^-(c)) = \ds 1 + \frac{12}{7n_i}\cdot \prod_{p \mid c} \left(1 + \frac{p}{p + 1}\right);$
\item[\rm (c)] If $9 \mid c$, then $\Avg^{(i)}(\Cl_3^-(c)) = \ds 1 + \frac{30}{7n_i}\cdot\prod_{p \mid c} \left(1 + \frac{p}{p + 1}\right).$

\end{enumerate}
\end{proposition}

\begin{proof}
Let $S_c$ denote the set of primes dividing $c$, and recall that the density of fundamental discriminants is:
\begin{equation}\label{quadcount}
	\ds \lim_{X\rightarrow \infty} \frac{\ds \sum_{0 < (-1)^i\Disc(K_2)< X} 1}{X} = \frac{1}{2\cdot\zeta(2)}.
	\end{equation}\\
	
(a) \ \ Assume $3 \nmid c$.   
 Recall that a $c$-valid cubic field $K_3$ has discriminant $\Disc(K_3) = df^2$ where $d$ is the discriminant of its quadratic resolvent field and $f \mid c$. Furthermore, it follows (for example, from Proposition 8.4.1(1) of \cite{cohen}) that a prime $p$ totally ramifies in $K_3$ if and only if $p \mid f$. Proposition \ref{decomp} in conjunction with \eqref{quadcount} therefore implies:
   	 $$\Avg^{(i)}(\Cl_3^-(c)) =
    		  1 + 4\cdot \zeta(2) \cdot \lim_{X\rightarrow \infty}\frac{{\ds \sum_{S \subseteq S_c}} N_3^{(i)}(\cK_S,X\cdot\prod_{p \in S} p^2)}{X},$$\\
where $S_c$ is equal to the set of primes dividing $c$.
By Theorem \ref{simple}(a), we conclude that
	\begin{eqnarray*}
		\ds \Avg^{(i)}(\Cl_3^-(c)) &= &\ds 1 + 4 \cdot \zeta(2) \cdot {\ds {\sum_{S\subseteq S_c}}} \ds\left(\prod_{p \in S} p^2 \cdot \frac{1}{n_i}\cdot\frac{1}{2\cdot\zeta(2)} \cdot \prod_{p \in S} \frac{1}{p(p+1)}\right)\\
		&  =& \ds 1 + \frac{2}{n_i}\cdot \prod_{p \in S_c} \left(1 + \frac{p}{p + 1}\right). 
	\end{eqnarray*}
This proves (a). We skip the proof of (b) as it is very similar to the proof of (c).

\medskip
\noindent (c) \ \ Assume $9 \mid c$. By Proposition \ref{decomp} and \eqref{quadcount} (in conjunction with Proposition 8.4.1(1) of \cite{cohen}), we have that
\begin{eqnarray*}
\Avg^{(i)}(\Cl_3^-(c)) &=& 1 + 4\cdot\zeta(2)\cdot\left( \lim_{X\rightarrow \infty}\frac{{\ds \sum_{S \subseteq S_c\smallsetminus\{3\}}} N_3^{(i)}(\cK_S,X\cdot\ds\prod_{p \in S} p^2) + \  N_3^{(i)}(\cK_{S\cup\{3\}}^{(3)},X \cdot \ds\prod_{p \in S \cup \{3\}} p^2)}{X}\right.\\
& & \quad \qquad \qquad + \left. \lim_{X\rightarrow\infty} \frac{\ds\sum_{S\subseteq S_c \smallsetminus \{3\}}N_3^{(i)}(\cK_{S \cup \{3\}}^{(9)} , 9X \cdot \ds \prod_{p \in S \cup \{3\}} p^2)}{X}\right),
\end{eqnarray*}
where $S_c$ again denotes the set of primes dividing $c$. Theorem \ref{simple}(a) and (c) then imply that $\Avg^{(i)}(\Cl_3^-(c))$ is equal to
\begin{eqnarray*}
\ds  
& & 1 + 4\cdot\zeta(2) \cdot \sum_{S \subseteq S_c \smallsetminus \{3\}} \left(\ \frac{1}{n_i}\cdot\frac{1}{2\cdot\zeta(2)} \cdot\prod_{p \in S}  \frac{p}{p+1}  \ + \  \frac{1}{n_i}\cdot\frac{1}{3\cdot \zeta(2)} \cdot\prod_{p \in S \cup \{3\}} \frac{p}{p+1} + \ \frac{3}{n_i}\cdot\frac{1}{2\cdot\zeta(2)} \cdot\prod_{p \in S \cup \{3\}} \frac{p}{p+1}\right) \\
& = & 1 + \frac{30}{7n_i} \cdot \prod_{p \in S_c} \left(1 + \frac{p}{p+1}\right).
\end{eqnarray*}
\end{proof}

\subsection{The asymptotic number of quadratic fields in certain acceptable families}

 In order to vary the family of quadratic fields we average over, we must first determine the asymptotics of these families. We first describe the asymptotic number of discriminants of quadratic fields that are relatively prime to a fixed integer. 

\begin{lemma}\label{coprime} Let $c$ be a positive integer. 
$$\lim_{X\rightarrow \infty} \frac{{\ds \sum_{{\scriptstyle(\Disc(K_2),c) = 1}\atop{\scriptstyle0 < (-1)^i\Disc(K_2) < X}} 1\ \ }}{X} = \frac{1}{2\cdot\zeta(2)} \cdot \prod_{p \mid c} \frac{p}{p+1}.$$
\end{lemma}
\begin{proof} By Proposition 2.2 and (4.2) in \cite{formulae}, we have that the number of real (resp. imaginary) quadratic fields that are unramified away from $c$ is asymptotically equal to
	$$\frac{1}{2}\cdot \prod_{p \nmid c} \left(\frac{p-1}{p}\cdot \left(1 + \frac{1}{p}\right)\right) \cdot \prod_{p \mid c} \left(\frac{p-1}{p}\cdot 1\right) \cdot X = \frac{1}{2\cdot\zeta(2)} \cdot \left(\prod_{p \mid c} \frac{p}{p+1}\right)\cdot X.$$
\end{proof}

Next, we determine the asymptotic number of quadratic fields with prescribed splitting at a finite number of primes.

\begin{lemma}\label{S} Let $\underline{\cS} = (S_+,S_-,S_0)$ be disjoint sets of primes, and let $\cK_2(\underline{\cS})$ denote the set of isomorphism classes of quadratic fields $K_2$ such that any prime $p \in S_+$ splits in $K_2$, any prime $p \in S_-$ remains inert in $K_2$, and any prime $p \in S_0$ ramifies in $K_2$. We then have:

$$\lim_{X\rightarrow \infty} \frac{{\ds \sum_{\scriptstyle K_2 \in \cK_2(\underline{\cS})\atop{\scriptstyle0 < (-1)^i\Disc(K_2) < X}} 1\ \ }}{X} = \frac{1}{2\cdot\zeta(2)} \cdot \prod_{p \in S_0} \frac{1}{p+1} \cdot \prod_{p \in S_{\pm}} \frac{p}{2(p+1)},$$
where $S_{\pm} = S_+ \cup S_-$.
\end{lemma}

\begin{proof} Similarly, by Proposition 2.2 and (4.2) in \cite{formulae}, we have that the asymptotic number of real (resp. imaginary) quadratic fields in $\cK_2(\underline{\cS})$ with (absolute) discriminant bounded by $X$ is
	$$\frac{1}{2} \cdot \prod_{p \notin \underline{\cS}} \left(\frac{p-1}{p} \cdot\left(1 + \frac{1}{p}\right)\right)\cdot \prod_{p \in S_{\pm}} \left(\frac{p-1}{p} \cdot \frac{1}{2}\right) \cdot \prod_{p \in S_0} \left(\frac{p-1}{p} \cdot \frac{1}{p}\right)\cdot X = \frac{1}{2\cdot\zeta(2)} \cdot \prod_{p \in S_0} \frac{1}{p+1} \cdot \prod_{p \in S_{\pm}} \frac{p}{2(p+1)}\cdot X$$

\end{proof}

\subsection{Averaging $\#\Cl^-_3(K_2,c)$ over quadratic fields unramified at $c$}
We vary over only those quadratic fields whose discriminants are coprime to the choice of fixed conductor. 

For shorthand, let $$\Avg^{(i)}_c(\Cl_3^-(c)) :=  \lim_{X \rightarrow \infty} \frac{\ds \sum_{{\scriptstyle(\Disc(K_2),c) = 1}\atop{\scriptstyle0 < (-1)^i\Disc(K_2) < X}} \#\Cl_{3}^-(K_2,c)}{\ds \sum_{{\scriptstyle(\Disc(K_2),c) = 1}\atop{\scriptstyle0 < (-1)^i\Disc(K_2) < X}} 1}.$$ 

\begin{proposition}\label{h-prime} Fix a positive integer $c$, and let $n$ denote the number of distinct primes dividing $c$. Recall that $n_0 = 6$ and $n_1 = 2$. Then
\begin{enumerate}
\item[\rm (a)] If $3 \nmid c$, then $\Avg^{(i)}_c(\Cl_3^-(c)) = 1 + 2\cdot\ds\frac{2^{n}}{n_i};$
\item[\rm (b)] If $3 \mid\mid c$, then $\Avg^{(i)}_c(\Cl_3^-(c)) = 1 + \ds\frac{2^n}{n_i};$
\item[\rm (c)] If $9 \mid c$, then $\Avg^{(i)}_c(\Cl_3^-(c)) = 1 + 3\cdot\ds\frac{2^n}{n_i}.$
\end{enumerate}
\end{proposition}
\begin{proof}
Let $S_c$ denote the set of primes dividing $c$. \\
(a) \ \ Assume $3 \nmid c$. 
 Proposition \ref{decomp} combined with Lemma \ref{coprime} implies that
   	 $$\Avg_c^{(i)}(\Cl_3^-(c)) =
    		  1 + 4\cdot\zeta(2) \cdot \left(\prod_{p\mid c} \frac{p+1}{p}\right) \cdot  \lim_{X\rightarrow \infty} \frac{{\ds \sum_{S \subseteq S_c}} N_3^{(i)}(\cK_{S,S_c} , X \cdot \prod_{p \in S} p^2)}{X}.$$
By Theorem \ref{simple}(d), we conclude that
	\begin{eqnarray*}
		\ds \Avg_c^{(i)}(\Cl_3^-(c)) &= &\ds 1 + 4\cdot\zeta(2)\cdot \prod_{p \mid c} \frac{p+1}{p} \cdot {\ds {\sum_{S\subseteq S_c}}} \ds\left(\frac{1}{n_i}\cdot\frac{1}{2\cdot\zeta(2)}\cdot\prod_{p \in S} \frac{p}{p+1} \cdot \prod_{p \in S_c \smallsetminus S} \frac{p}{p+1}\right)\\
		&  =& \ds 1 + \frac{2^{n+1}}{n_i}. 
	\end{eqnarray*}\\
(b) \ \ Note that for non-Galois cubic fields $K_3$ that are totally ramified at $3$, $\Disc(K_3)$ is exactly divisible by $3^3$, $3^4$, or $3^5$, and in order for the quadratic resolvent $K_2$ of $K_3$ to have discriminant relatively prime to $3$, then $\Disc(K_3) = \Disc(K_2)f^2$ where either $3 \nmid f$ or $9 \mid\mid f$. Thus, if $3 \mid \mid c$, any quadratic field $K_2$ that is unramified at $3$ satisfies 
	$$\Cl^-(K_2,c) = \Cl^-(K_2,\textstyle\frac{c}{3}).$$
Thus, Proposition \ref{decomp}, Lemma \ref{coprime}, and Theorem \ref{simple}(d) together imply that
	\begin{eqnarray*}
		\ds \Avg_c^{(i)}(\Cl_3^-(c)) &= &1 + 4\cdot\zeta(2) \cdot \left(\prod_{p\mid c} \frac{p+1}{p}\right) \cdot \lim_{X\rightarrow \infty}\frac{{\ds \sum_{S \subseteq S_c\smallsetminus \{3\}}} N_3^{(i)}(\cK_{S,S_c}, X\cdot \prod_{p \in S} p^2)}{X}\\
				&  =& \ds 1 + \frac{2^{n}}{n_i}. 
	\end{eqnarray*}\\	
(c) \ \ If $9 \mid c$, by Proposition \ref{decomp} and Lemma \ref{coprime}, we have
\begin{eqnarray*}\Avg^{(i)}_c(\Cl_3^-(c))  &=&1 + 4\cdot \zeta(2) \cdot
    		  \left(\prod_{p\mid c} \frac{p+1}{p}\right) \cdot  \left(\lim_{X\rightarrow \infty} \frac{{\ds \sum_{S \subseteq S_c\smallsetminus \{3\}}} N_{3}^{(i)}( \cK_{S,S_c} ,X \cdot \prod_{p \in S} p^2)}{X}\right. \\
 & & \qquad \qquad  \qquad +\left. \lim_{X\rightarrow \infty} \frac{\ds\sum_{S\subseteq S_c \smallsetminus\{3\}} N_3^{(i)}(\cK^{(9)}_{S\cup\{3\},S_c} , 9X \cdot \prod_{p \in S\cup\{3\}} p^2)}{X}\right).
		  \end{eqnarray*}
Finally, by Theorem \ref{simple}(d) and (e), we conclude that
	$$\Avg^{(i)}_c(\Cl_3^-(c)) = 1 + 3\cdot \frac{2^n}{n_i}.$$
\end{proof}

\subsection{Averaging $\#\Cl^-_3(K_2,c)$ over quadratic fields with prescribed splitting at a finite number of primes}

Next, we vary over quadratic fields in $\cK_2(\underline{\cS})$. Let $\underline{\cS} = (S_+,S_-,S_0)$ be disjoint sets of primes, and recall that $\cK_2(\underline{\cS})$ denotes the set of isomorphism classes of quadratic fields $K_2$ such that any prime $p \in S_+$ splits in $K_2$, any prime $p \in S_-$ remains inert in $K_2$, and any prime $p \in S_0$ ramifies in $K_2$. Recall that we set $S_{\pm} = S_+ \cup S_-$.

For shorthand, let $$\Avg^{(i)}_{\underline{\cS}}(\Cl_3^-(c)) :=  \lim_{X \rightarrow \infty} \frac{\ds \sum_{{\scriptstyle K_2 \in \cK_2(\underline{\cS})}\atop{\scriptstyle0 < (-1)^i\Disc(K_2) < X}} \#\Cl_{3}^-(K_2,c)}{\ds \sum_{{\scriptstyle K_2 \in \cK_2(\underline{\cS})}\atop{\scriptstyle0 < (-1)^i\Disc(K_2) < X}} 1}.$$ 

\begin{proposition}\label{h-S} Fix a positive integer $c$ coprime to $3$, and let $\underline{\cS} = (S_+,S_-,S_0)$ be disjoint sets of primes such that any prime $p \mid c$ is not contained in $S_0$. If $S_{+}^{\rm good}$ $($respectively, $S_-^{\rm good})$ denote the subset consisting of all primes $p$ in $S_{+}$ $($resp., in $S_-)$ that are congruent to $1 \bmod 3$ $($resp., $2 \bmod 3)$ and $p \mid c$. We then have 
	$$\Avg_{\underline{\cS}}^{(i)}(\Cl_3^-(c)) =  1 + \frac{2}{n_i}\cdot 3^{\#(S^{\rm good}_+ \cup S^{\rm good}_-)}\cdot\prod_{{\scriptstyle p \mid c}\atop{\scriptstyle p \notin S_{\pm}}} \left(1 + \frac{p}{p + 1}\right)$$
\end{proposition}
\begin{proof}
In order to determine $\Avg_{\underline{\cS}}^{(i)}(\Cl_3^-(c))$, by Proposition \ref{decomp}, we must compute the asymptotic number of non-cyclic $c$-valid cubic fields whose quadratic resolvent field is in $\cK_2(\underline{\cS})$. If $p \in S_+$ and so $K_2 \in \cK_2(\underline{\cS})$ splits at $p$, then in any non-cyclic $c$-valid cubic field $K_3$ with quadratic resolvent $K_2$, $p$ remains inert in $K_3$, $p$ splits completely, or $p$ is totally ramified. Additionally, if $p \in S_-$, then $p\cO_{K_3}$ either decomposes into a product of exactly two distinct prime ideals or totally ramifies. Finally, we recall that a prime $p$ totally ramifies in $K_3$ with discriminant $\Disc(K_2)f^2$ if and only if $p \mid f$.  

Let $S_c$ denote the set of primes dividing $c$, and recall that by assumption, $S_c \cap S_0 = \emptyset$. Additionally, let $S_c^{\rm good}$ denote the set of all primes $p \mid c$ such that if $p \in S_+$ (respectively, if $p \in S_-$), then $p$ is congruent to $1 \bmod 3$ (resp., to $2 \bmod 3$), and set $c_{\rm good} = \prod_{p \in S_c^{\rm good}} p$. Proposition \ref{decomp} combined with Lemma \ref{S} implies that
   	 $$\Avg_{\underline{\cS}}^{(i)}(\Cl_3^-(c_{\rm good})) =
    		  1 + 4\cdot\zeta(2) \cdot \left(\prod_{p \in S_{\pm}} \frac{2(p+1)}{p}\right)\cdot\left(\prod_{p \in S_0} p+1 \right)\cdot  \lim_{X\rightarrow \infty} \frac{{\ds \sum_{S \subseteq S^{\rm good}_c}} N_3^{(i)}(\cK_{3}^S(\underline{\cS}) , X \cdot \prod_{p \in S} p^2)}{X}.$$
By Theorem \ref{simple}(a), we conclude that $\Avg_{\underline{\cS}}^{(i)}(\Cl_3^-(c_{\rm good}))$ is equal to 
\begin{eqnarray*}
&& 1 + 4\cdot\zeta(2)\cdot \left(\prod_{p \in S_{\pm}} \frac{2(p+1)}{p}\right) \cdot {\ds {\sum_{S\subseteq S^{\rm good}_c}}} \ds\left(\frac{1}{n_i}\cdot\frac{1}{2\cdot\zeta(2)} \cdot \prod_{p \in S}  \frac{p}{p+1} \cdot \prod_{p \in S_{\pm}\smallsetminus (S \cap S_{\pm})} \frac{p}{2(p+1)}\right)\\
		&  =& \ds 1 + \frac{2}{n_i}\cdot \sum_{S \subseteq S^{\rm good}_c} \left(\prod_{p \in S\smallsetminus (S \cap S_{\pm})} \frac{p}{p+1}\cdot\prod_{p \in S \cap S_{\pm}} 2\right)\\\\
		&=& \ds 1 + \frac{2}{n_i}\cdot 3^{\#(S^{\rm good}_c\cap S_{\pm})}\cdot \prod_{p \in S_c^{\rm good} \smallsetminus (S_c^{\rm good} \cap S_{\pm})} \left(1 + \frac{p}{p+1}\right).
	\end{eqnarray*}
We finish the proof by verifying that $\Avg_{\underline{\cS}}^{(i)}(\Cl_3^-(c_{\rm good})) = \Avg_{\underline{\cS}}^{(i)}(\Cl_3^-(c)).$ Indeed, the 3-torsion subgroups of ray class groups of conductor $\frac{c}{c_{\rm good}}$ are trivial for quadratic fields in $\cK_2(\underline{\cS})$. If a prime $p \mid c$ and $p \nmid c_{\rm good}$, then either $p \in S_+$ and $p \equiv 2 \bmod 3$ or $p \in S_-$ and $p \equiv 1 \bmod 3$. In both of these cases, there are no cubic extensions of a quadratic field in $\cK_2(\underline{\cS})$ that are totally ramified at $p$. 
\end{proof}

\begin{remark}\label{afterh-S} Let $c$ be a positive integer exactly divisible by $3$. If $\underline{\cS} = (S_+, S_-, S_0)$ is as in Proposition \ref{h-S}, but we further assume $3 \notin S_+ \cup S_-$, then
\begin{eqnarray*}
\Avg^{(i)}_{\underline{\cS}}(\Cl_3^-(c)) &=&  1 + \frac{12}{7n_i} \cdot 3^{\#(S^{\rm good}_+ \cup S^{\rm good}_-)}\cdot\prod_{{\scriptstyle p \mid c}\atop{\scriptstyle p \notin S_{\pm}}} \left(1 + \frac{p}{p + 1}\right),
\end{eqnarray*}
where $S^{\rm good}_{\pm}$ are as in Proposition \ref{h-S}.
\end{remark}

\section{Proofs of Theorems \ref{main}-\ref{2} and Corollary \ref{3}}

We put together the results of the previous sections in order to conclude Theorems \ref{main} and \ref{2}, as well as Corollary \ref{3}. For Theorem \ref{main}, we will first allow $K_2$ to vary over all quadratic fields of bounded discriminant and use the combination of Propositions \ref{decomp}, \ref{h+}, and \ref{h-}.  Afterwards, we only vary over the discriminants of quadratic fields that are coprime to the fixed conductor, and we combine Propositions \ref{decomp}, \ref{h+}, and \ref{h-prime} in order to conclude Theorem \ref{2}. Corollary \ref{3} is then derived from a generalization of Theorem \ref{main}.

\subsection{Proof of Theorem \ref{main}}

We combine Propositions \ref{h-} and \ref{h+} of the previous sections to prove Theorem \ref{main}. Let $c$ be an integer, and assume that there are $m$ primes dividing $c$ that are congruent to $1 \bmod{3}$. Define $k$ to satisfy $3^k \mid \mid c$, and recall that $\Cl^+_{3}(K_2,c)$ only depends on $m$ and $k$. It is independent of the choice of quadratic field, so in particular, we have by Proposition \ref{h+},
	\begin{eqnarray*} \lim_{X \rightarrow \infty} \frac{ \ds\sum_{0 < (-1)^i\Disc(K_2) < X} \#\Cl_{3}(K_2,c)}{\ds \sum_{0 < (-1)^i\Disc(K_2) < X} 1} &=& \#\Cl^+_{3}(K_2,c) \cdot \lim_{X \rightarrow \infty} \frac{\ds \sum_{0 < (-1)^i\Disc(K_2) < X} \#\Cl^-_{3}(K_2,c)}{\ds\sum_{0 < (-1)^i\Disc(K_2) < X} 1}. 
	 \end{eqnarray*}
	 
\noindent We conclude by Propositions \ref{h+} and \ref{h-} in conjunction with \eqref{pmdecomp} that
	\begin{equation}\label{thm1}
	\lim_{X \rightarrow \infty} \frac{\ds \sum_{0 < (-1)^i\Disc(K_2) < X} \#\Cl_{3}(K_2,c)}{\ds \sum_{0 < (-1)^i\Disc(K_2) < X} 1} =  
	\begin{cases} 3^m\cdot\left(1 + \ds\frac{2}{n_i}\cdot \ds\prod_{p \mid c} \left(1 + \frac{p}{p + 1} \right)\right) & \mbox{ if $k = 0$; } \\
	3^{m} \cdot \left(1 + \ds\frac{12}{7n_i}\cdot \ds\prod_{p \mid c} \left(1 + \frac{p}{p + 1} \right)\right) & \mbox{ if $k = 1$; } \\
	3^{m+1} \cdot \left(1 + \ds\frac{30}{7n_i} \cdot\ds \prod_{p \mid c} \left(1 + \frac{p}{p + 1} \right)\right) & \mbox{ if $k \geq 2$. } 
	\end{cases}
	\end{equation}
	
\subsection{Proof of Theorem \ref{2}}
In order to compute the average number of $3$-torsion elements in ray class groups of fixed conductor of quadratic fields with discriminant that is both bounded and coprime to the choice of conductor, we combine Propositions \ref{h+} and \ref{h-prime}. Let $c$ be an integer, and assume that there are $n$ distinct primes dividing $c$, $m$ of which are congruent to $1 \bmod 3$. Define $k$ to satisfy $3^k \mid \mid c$. 
By Proposition 3.1 and \eqref{pmdecomp},
	\begin{equation*} \lim_{X \rightarrow \infty} \frac{ \ds\sum_{{\scriptstyle(\Disc(K_2),c) = 1}\atop{\scriptstyle0 < (-1)^i\Disc(K_2) < X}} \#\Cl_{3}(K_2,c)}{\ds \sum_{{\scriptstyle(\Disc(K_2),c) = 1}\atop{\scriptstyle0 < (-1)^i\Disc(K_2) < X}} 1} = \#\Cl^+_{3}(K_2,c) \cdot \lim_{X \rightarrow \infty} \frac{\ds \sum_{{\scriptstyle(\Disc(K_2),c) = 1}\atop{\scriptstyle0 < (-1)^i\Disc(K_2) < X}} \#\Cl^-_{3}(K_2,c)}{\ds\sum_{{\scriptstyle(\Disc(K_2),c) = 1}\atop{\scriptstyle0 < (-1)^i\Disc(K_2) < X}} 1}. \end{equation*}
Combining with Proposition \ref{h-prime}, we obtain
\begin{equation}
\lim_{X \rightarrow \infty} \frac{ \ds\sum_{{\scriptstyle(\Disc(K_2),c) = 1}\atop{\scriptstyle0 < (-1)^i\Disc(K_2) < X}} \#\Cl_{3}(K_2,c)}{\ds \sum_{{\scriptstyle(\Disc(K_2),c) = 1}\atop{\scriptstyle0 < (-1)^i\Disc(K_2) < X}} 1} = 
\begin{cases} 3^m\cdot\left(1 + \ds\frac{2^{n+1}}{n_i}\right) & \mbox{ if $k = 0$,} \vspace{.1in}\\
	3^{m} \cdot \left(1 + \ds\frac{2^n}{n_i}\right) & \mbox{ if $k = 1$, and } \vspace{.1in}\\
	3^{m+1} \cdot \left(1 + \ds3\cdot\frac{2^n}{n_i}\right) & \mbox{ if $k \geq 2$}.
	\end{cases}
	\end{equation}
\subsection{Generalizing Theorem \ref{main}(a) and the proof of Corollary \ref{3}}

Before turning to the proof of Corollary \ref{3}, we first generalize Theorem \ref{main}(a) when $(6,c) = 1$. 

\begin{theorem}\label{maingen} Let $c$ be an integer coprime to $3$, and let $\underline{\cS} = (S_+,S_-,S_0)$ be a disjoint set of primes such that no prime $p \mid c$ is contained in $S_0$. Let $m$ denote the number of primes $p \mid c$ that are congruent to $1 \bmod 3$, and let $S_+^{\rm good}$ $($respectively, $S_-^{\rm good})$ denote the subset of primes $p \mid c$ that are contained in $S_+$ $($resp., $S_-)$ and congruent to $1 \bmod 3$ $($resp., $2 \bmod 3)$.
\begin{enumerate}
\item[\rm (a)] The average size of the 3-torsion subgroups in ray class groups of conductor $c$ of real quadratic fields that are split at primes in $S_+$, inert at primes in $S_-$, and ramified at primes in $S_0$ is
	$$3^m\cdot\left(1 + \ds 3^{\#(S_+^{\rm good} \cup S_-^{\rm good})-1}\cdot\prod_{{\scriptstyle p \mid c}\atop{\scriptstyle p \notin S_+ \cup S_-}} \left(1 + \frac{p}{p+1}\right)\right)$$
when these quadratic fields are ordered by their discriminant. 
\item[\rm (b)] The average size of the 3-torsion subgroups in ray class groups of conductor $c$ of imaginary quadratic fields that are split at primes in $S_+$, inert at primes in $S_-$, and ramified at primes in $S_0$ is
	$$3^m\cdot\left(1 + \ds 3^{\#(S_+^{\rm good} \cup S_-^{\rm good})}\cdot\prod_{{\scriptstyle p \mid c}\atop{\scriptstyle p \notin S_+ \cup S_-}} \left(1 + \frac{p}{p+1}\right)\right)$$
when these quadratic fields are ordered by their discriminant.
\end{enumerate}
\end{theorem}

\begin{proof} In order to compute the average number of $3$-torsion elements in ray class groups of fixed conductor of quadratic fields with prescribed splitting, we combine Propositions \ref{h+} and \ref{h-S} using \eqref{pmdecomp}. We obtain:
	\begin{eqnarray*} \lim_{X \rightarrow \infty} \frac{ \ds\sum_{{\scriptstyle K_2 \in \cK_2(\underline{\cS})}\atop{\scriptstyle0 < (-1)^i\Disc(K_2) < X}} \#\Cl_{3}(K_2,c)}{\ds \sum_{{\scriptstyle K_2 \in \cK_2(\underline{\cS})}\atop{\scriptstyle0 < (-1)^i\Disc(K_2) < X}} 1} &=& \#\Cl^+_{3}(K_2,c) \cdot \lim_{X \rightarrow \infty} \frac{\ds \sum_{{\scriptstyle K_2 \in \cK_2(\underline{\cS})}\atop{\scriptstyle0 < (-1)^i\Disc(K_2) < X}} \#\Cl^-_{3}(K_2,c)}{\ds\sum_{{\scriptstyle K_2 \in \cK_2(\underline{\cS})}\atop{\scriptstyle0 < (-1)^i\Disc(K_2) < X}} 1}\\
&=&  3^m\cdot\left(1 + \ds\frac{2}{n_i}\cdot 3^{\#(S_+^{\rm good} \cup S_-^{\rm good})}\prod_{{\scriptstyle p \mid c}\atop{\scriptstyle p \notin S_+ \cup S_-}} \left(1 + \frac{p}{p+1}\right)\right).
	\end{eqnarray*}
\end{proof}

The above theorem (along with Remark \ref{afterh-S}) directly implies that as long as $S$ is disjoint from $S_+ \cup S_-$, the mean size of the 3-torsion subgroups in ray class groups of conductor $c$ are independent of the family $\cK_2(\underline{\cS})$ of quadratic fields one averages over. More precisely:

\begin{corollary}\label{maincor}Let $c$ be an integer such that $9 \nmid c$, and let $\underline{\cS} = (S_+,S_-,S_0)$ be a disjoint set of primes such that no prime $p \mid c$ is contained in $S_+ \cup S_- \cup S_0$. If  $m$ denotes the number of primes $p \mid c$ that are congruent to $1 \bmod 3$, then the average size of the 3-torsion subgroups in ray class groups of conductor $c$ of quadratic fields with $i$ pairs of complex embeddings that are split at primes in $S_+$, inert at primes in $S_-$, and ramifies in $S_0$ is equal to
	\begin{eqnarray*} 3^m\cdot \left(1 + \ds\frac{2}{n_i}\cdot \prod_{p \mid c} \left(1 + \frac{p}{p+1}\right)\right) & &\mbox{ if $3 \nmid c$,} \\
	3^m\cdot \left(1 + \ds\frac{12}{7n_i}\cdot \prod_{p \mid c} \left(1 + \frac{p}{p+1}\right)\right) && \mbox{ if $3 \mid c$.}\end{eqnarray*}
when these quadratic fields are ordered by discriminant.
\end{corollary}
This allows for the generalization of Theorem~3 given in Corollary \ref{3}, whose proof we turn to next. 
\medskip

\noindent {\em Proof of Corollary \ref{3}}. We use Corollary \ref{maincor} to compute lower bounds for the proportion $P_i(\underline{\cS},c)$ of quadratic fields in $\cK_2(\underline{\cS})$ with $i$ pairs of complex embeddings whose ray class groups of conductor $c$ have trivial $3$-torsion subgroup. 

We assume $(3,c) = 1$. If $m$ denotes the distinct number of primes dividing $c$ that are congruent to $1 \bmod 3$, we have by Theorem \ref{maingen},
	\begin{equation*}
	3^m\cdot\left(1 + \frac{2}{n_i}\cdot \prod_{p \mid c} \left( 1 + \frac{p}{p + 1} \right)\right) \geq  1 \cdot P_i(\underline{\cS},c) + 3 \cdot (1-P_i(\underline{\cS},c)).\label{proportion}
	\end{equation*}
Hence, $P_i(\underline{\cS},c) > 0$ if and only if $m = 0$ and  
$$n_i > \prod_{p \mid c}\left(1 + \frac{p}{p+1}\right).$$
Thus, we conclude automatically that for any conductor $c$ of the form $c=p$ where $p \equiv 2 \bmod 3$, a positive proportion of real (resp.\ imaginary) quadratic fields have trivial $3$-torsion subgroup in their ray class groups of conductor $c$. Additionally, for any conductor $c$ of the form $c = p_1p_2$ where $p_i \equiv 2 \bmod 3$, we see that a positive proportion of real quadratic fields have trivial $3$-torsion subgroup in their ray class groups of conductor $c$.

If $3 \mid \mid c$, and $m$ denotes the distinct number of primes dividing $c$ that are congruent to $1 \bmod 3$, we similarly have 
that $P(i,c) > 0$ if and only if $ m = 0$ and $$\frac{7n_i}{6} > \prod_{{p \mid c}} \left(1 + \frac{p}{p+1}\right).$$

Thus, for real quadratic fields, if $c$ is a product of $3$ and a prime $p$ which is congruent to $2 \bmod 3$, a positive proportion of real quadratic fields have trivial $3$-torsion subgroup in their ray class groups of conductor $c$. Additionally, a positive proportion of imaginary quadratic fields have trivial $3$-torsion subgroup in their ray class groups of conductor 3. \hfill $\square$

\begin{remark} Similarly, one can show that if $3 \nmid c$,
at least $50$\% of real quadratic fields have trivial $3$-torsion subgroup in their ray class groups of prime conductor $c \equiv 2 \bmod{3}$. If $3 \mid\mid c$,
at least $50$\% of real quadratic fields have trivial $3$-torsion subgroup in their ray class groups of conductor $3$ or $3p$ where $p \equiv 2 \bmod{3}$. 
\end{remark}

\section{Second main term and the proof of Theorem \ref{4}}
To compute the second main term for the mean number of $3$-torsion elements in ray class groups of quadratic fields of bounded discriminant, we use a refinement of Theorem \ref{simple}. For any set of primes $S$ not containing $3$, recall that $\cK_S$ denotes the set of isomorphism classes of cubic fields that are totally ramified {\em exactly} at the primes $p \in S$.  We first introduce some notation from \cite{simple} and \cite{TaniguchiThorne}. For a free $\bZ_p$-module $M$, define $M^{\rm Prim} \subset M$ by $M^{\rm Prim} := M \smallsetminus \{pM\}.$ Also, for any element $x$ in a cubic order, let $i(x) := [R:\bZ_p[x]]$. As in the proof of Theorem \ref{simple}, let $\Sigma^{S}$ denote the set of all isomorphism classes of rings of integers of cubic fields in $\cK_S$. Then, $\Sigma^{S}$ is {\em strongly acceptable} as defined in \cite{simple}. Thus, if $N_3^{(i)}(\Sigma^{S},X)$ denotes the number of cubic orders $R \in \Sigma^{S}$ satisfying $0 < (-1)^i\Disc(R) < X$, Theorem 1.3 of \cite{TaniguchiThorne} determines the asymptotic count with two main terms: 

\begin{equation}
\begin{array}{rcl}
N_3^{(i)}(\Sigma^S;X)\!\!\!&=&\!\!\!
\displaystyle{\frac{1}{2n_i}\cdot
\prod_p\Bigl(\frac{p-1}{p}\cdot\sum_{R\in\Sigma_p}
\frac{1}{\Disc_p(R)}\cdot\frac1{\#\Aut(R)}\Bigr)}
\cdot X \vspace{.1in} \\[.1in]  && + \,\,
\displaystyle\frac{c_2^{(i)}}{\zeta(2)}\cdot
\prod_p\Bigl((1-p^{-1/3})\cdot\sum_{R\in\Sigma_p}\frac{1}{\Disc_p(R)}\cdot\frac1{\#\Aut(R)}\int_{(R/\Z_p)^{{\rm Prim}}}i(x)^{2/3}dx\Bigr)
\cdot X^{5/6}\,\,\vspace{.1in}\\
& & + \,\,O_{\epsilon}(X^{5/6-7/138+\epsilon})\:\!,
\end{array}\end{equation}
where $dx$ assigns measure $1$ to $(R/\Z_p)^{{\rm Prim}}$, and additionally, 
	$$c_2^{(i)}=\begin{cases}
    \displaystyle{\frac{\sqrt{3}\zeta(2/3)\Gamma(1/3)(2\pi)^{1/3}}{30\Gamma(2/3)}}& \textrm{ if $i = 0,$}\vspace{.1in}\\
    \displaystyle{\,\,\;\;\;\;\frac{\,\zeta(2/3)\Gamma(1/3)(2\pi)^{1/3}}{10\Gamma(2/3)}}& \textrm{ if $i = 1;$ }\\
\end{cases} \qquad \mbox{ and } \qquad \Sigma_p = \begin{cases}  \Sigma^{\rm ntr}_p & \mbox{ if } p \notin S, \\
 \Sigma_p^{\rm tr} & \mbox{ if } p \in S.
\end{cases}$$
(Recall that for any prime $p$, $\Sigma_p^{\rm tr}$ denotes the set of all isomorphism classes of maximal cubic orders over $\bZ_p$ that are totally ramified, and $\Sigma^{\rm ntr}_p$ denotes the set of all isomorphism classes of maximal cubic orders over $\bZ_p$ that are not totally ramified.)

In order to compute the second main term's constant, we combine Table 1, Lemma 28, and Lemma 37 in \cite{simple} to determine
$$\sum_{R \in \Sigma_p} \frac{1}{\Disc_p(R)}\cdot\frac{1}{\#\Aut(R)} \int_{(R/\bZ)^{\rm Prim}} i(x)^{2/3} dx = \begin{cases} \ds\frac{1}{p(p+1)} + \frac{1}{p^{4/3}(p+1)} & \mbox{ if $p \in S$, and } \vspace{.1in}\\
\ds \frac{p^{1/3}}{p^{1/3} - 1} - \frac{p^{2/3} + p^{1/3}}{p(p+1)(p^{1/3} - 1)} & \mbox{ if $p \notin S$.} \end{cases}$$

We then calculate that $\ds\prod_p\left(1-p^{-1/3}\right)\cdot \ds\prod_{R \in \Sigma_p} \frac{1}{\Disc_p(R)}\cdot \frac{1}{\#\Aut(R)}\int_{(R/\bZ)^{\rm Prim}}i(x)^{2/3}dx$ is equal to
	\begin{eqnarray*}
	& & \prod_{p} \left(1 -{p^{-1/3}}\right) \left(\frac{p^{1/3}}{p^{1/3} - 1} - \frac{p^{2/3} + p^{1/3}}{p(p+1)(p^{1/3} - 1)}\right) \cdot \prod_{p \in S} \frac{\ds\frac{1}{p(p+1)} + \frac{1}{p^{4/3}(p+1)}}{\ds\frac{p^{1/3}}{p^{1/3} - 1} - \frac{p^{2/3} + p^{1/3}}{p(p+1)(p^{1/3} - 1)}} \\
&=& \prod_p \left(1- {p^{-1/3}}	\right)\cdot\left(\frac{p^{1/3}p(p+1) - p^{2/3} - p^{1/3}}{p(p+1)(p^{1/3} - 1)}\right) \cdot \prod_{p \in S} \frac{p^{2/3} - 1}{p^{8/3} + p^{5/3} - p - p^{2/3}} \\ 
&=& \prod_p 1-\frac{p^{1/3} + 1}{p(p+1)}\cdot \prod_{p \in S} \frac{1}{p(p+1)}\cdot \frac{1 - p^{-2/3}}{1 -\ds \frac{p^{1/3}+1}{p(p+1)}}.\end{eqnarray*}
We can thus conclude the following refinement of Theorem \ref{simple}. 

\begin{theorem}\label{secondsimple}
Let $S$ denote a set of primes not containing $3$, and let $n_i = |\Aut(\bR^{3-2i} \oplus \bC^i)|$ for $i = 0$ or $1$. Let $\cK_S$ denote the set of isomorphism classes of cubic fields that are totally ramified {\em exactly} at the primes $p \in S$. 
\begin{eqnarray*}
N^{(i)}_3(\cK_S,X) &=& \frac{1}{n_i}\cdot\frac{1}{2\cdot\zeta(2)}\cdot \prod_{p \in S} \frac{1}{p(p+1)} \cdot X  \\ & & + \  \frac{c_2^{(i)}}{\zeta(2)} \cdot \prod_{p} \left(1 - \frac{p^{1/3} + 1}{p(p+1)}\right) \prod_{p \in S} \left(\frac{1}{p(p+1)}\cdot \frac{1 - p^{-2/3}}{1-\frac{p^{1/3}+1}{p(p+1)}}\right)\cdot X^{5/6} \\ & & + \ O_\epsilon(X^{5/6 - 7/138 + \epsilon}),
\end{eqnarray*}
where $$c_2^{(i)}=\begin{cases}  
  \displaystyle{\frac{\sqrt{3}\zeta(2/3)\Gamma(1/3)(2\pi)^{1/3}}{30\Gamma(2/3)}}& \textrm{ if $i = 0,$ and }\vspace{.1in}\\
  \displaystyle{\,\,\;\;\;\;\frac{\,\zeta(2/3)\Gamma(1/3)(2\pi)^{1/3}}{10\Gamma(2/3)}}& \textrm{ if $i = 1.$ }\\
\end{cases}$$
\end{theorem}
We are now ready to prove Theorem \ref{4}. Let $c$ be a positive integer coprime to $3$, and let $S_c$ denote the set of primes dividing $c$. Proposition \ref{decomp} combined with Theorem \ref{secondsimple} implies that
   	 \begin{eqnarray*}
	 \sum_{0 < (-1)^i\Disc(K_2)< X} \#\Cl_3^-(K_2,c) &=& 
    		  1 + 2\left. \cdot {{\ds \sum_{S \subseteq S_c}} N_3^{(i)}(\cK_S,X\cdot\prod_{p \in S} p^2)}\right. \\
		 & = &  1 + 2\cdot\left[\frac{1}{n_i}\cdot\frac{1}{2\cdot\zeta(2)}\cdot \sum_{S \subseteq S_c} \left(\prod_{p \in S} \frac{1}{p(p+1)} \cdot X \cdot \prod_{p \in S} p^2 \right.\right. \\
		 & & + \left.\frac{c_2^{(i)}}{\zeta(2)} \cdot \prod_{p} \left(1 - \frac{p^{1/3} + 1}{p(p+1)}\right) \cdot\prod_{p \in S} \bigg(\frac{1}{p(p+1)}\cdot \frac{1 - p^{-2/3}}{1-\frac{p^{1/3}+1}{p(p+1)}}\bigg)\cdot X^{5/6} \cdot \prod_{p \in S} p^{5/3}\right. \\
		 & & + \ O_{\epsilon}(X^{5/6 - 7/138 + \epsilon})\Bigg)\Bigg]. 
		 \end{eqnarray*}
Simplifying, we conclude
\begin{eqnarray*}
		\sum_{0 < (-1)^i\Disc(K_2)< X} \#\Cl_3^-(K_2,c) & = & 1 + 2\cdot \left[ \frac{1}{n_i}\cdot \prod_{p \in S}\left( 1 + \frac{p}{p+1}\right)\right.\cdot \sum_{0 < (-1)^i\Disc(K_2) < X} 1 \\
		 & & {+ \ \frac{c_2^{(i)}}{\zeta(2)}\cdot\prod_p\left(1 - \frac{p^{1/3} + 1}{p(p+1)}\right)\cdot \prod_{p \in S} \bigg(1 + \frac{p(1-p^{1/3})}{1 -\frac{p(p+1)}{p^{1/3} + 1}}\bigg)}\cdot X^{5/6}\Bigg] \\
		 & & + \ O_{\epsilon,c}(X^{5/6 - 7/138 + \epsilon}). 
		 \end{eqnarray*}
Combining with Proposition \ref{h+} and \eqref{pmdecomp}, we deduce Theorem \ref{4}.

\section*{Acknowledgements}
The author would like to thank Manjul Bhargava for suggesting this problem and answering many questions. She would also like to thank Djordjo Milovic, Carlo Pagano, Arul Shankar, and Jacob Tsimerman for helpful discussions. The author was supported by a National Defense Science \& Engineering Fellowship and NSF Grant DMS-1502834.

\end{document}